\documentclass[a4paper,10pt]{article}
\usepackage{a4wide}
\usepackage[T1]{fontenc}		
\usepackage[utf8]{inputenc}

\usepackage{enumitem} 

\usepackage{color}
\usepackage{bbm}		 	
\usepackage{placeins}
\usepackage{amsmath}
\usepackage{amsfonts}
\usepackage{amssymb}
\usepackage{amsthm}
\usepackage{hyperref}
\usepackage{mathtools}
\usepackage[british]{babel}
\usepackage{natbib}
\bibpunct{(}{)}{;}{a}{}{,}

\usepackage[pdftex]{graphicx}
\usepackage{caption}
\usepackage{subcaption}
\usepackage{bm}

\newtheorem{theorem}{Theorem}
\newtheorem{proposition}{Proposition}
\newtheorem{definition}{Definition}
\newtheorem{corollary}{Corollary}
\newtheorem{lemma}{Lemma}
\newtheorem{remark}{Remark}

\newcommand{\R}{\mathbb{R}}

\newcommand{\N}{\mathbb{N}}

\newcommand{\cC}{\mathcal{C}}
\newcommand{\cH}{\mathcal{H}}

\newcommand{\fc}{\mathfrak{c}}
\newcommand{\fr}{\mathfrak{r}}

\newcommand{\iid}{ {\stackrel{i.i.d.}{\sim}} }
\newcommand{\Prb}{\mathbb{P}}
\newcommand{\EE}{\mathbb{E}}

\newcommand{\cN}{\mathcal{N}}

\newcommand{\cL}{\mathcal{L} }
\newcommand{\cS}{\mathcal{S} }
\newcommand{\cG}{\mathcal{G} }

\newcommand{\var}{\mbox{\rm var}}
\newcommand{\cov}{\mbox{\rm cov}}

\title{Simultaneous Confidence Bands for Functional Data Using the Gaussian Kinematic Formula}
\author{Fabian J.E. Telschow, Armin Schwartzman}

\begin{document}

\maketitle

\begin{abstract}
This article constructs simultaneous confidence bands (SCBs) for functional parameters using the Gaussian Kinematic formula of $t$-processes (tGKF). Although the tGKF relies on Gaussianity, we show that a central limit theorem (CLT) for the parameter of interest is enough to obtain asymptotically precise covering rates even for non-Gaussian processes. As a proof of concept we study the functional signal-plus-noise model and derive a CLT for an estimator of the Lipschitz-Killing curvatures, the only data dependent quantities in the tGKF SCBs. Extensions to discrete sampling with additive observation noise are discussed using scale space ideas from regression analysis. Here we provide sufficient conditions on the processes and kernels to obtain convergence of the functional scale space surface.

The theoretical work is accompanied by a simulation study comparing different methods to construct SCBs for the population mean. We show that the tGKF works well even for small sample sizes and only a Rademacher multiplier-$t$ bootstrap performs similarily well. For larger sample sizes the tGKF often outperforms the bootstrap methods and is computational faster. We apply the method to diffusion tensor imaging (DTI) fibers using a scale space approach for the difference of population means. R code is available in our Rpackage SCBfda.
\end{abstract}
\tableofcontents

\section{Introduction}
	  In the past three decades functional data analysis has received increasing interest due to the possibility of recording and storing data collected with high frequency and/or high resolution in time and space. Many different methods have been developed to study these particular data objects; for overviews of some recent developments in this fast growing field we refer the reader to the review articles \citet{Cuevas2014} and \citet{Wang2016}.
	  
	  Currently, the literature on construction of simultaneous confidence bands (SCBs) for functional parameters derived from repeated observations of functional processes is sparse. The existing methods basically split into two seperate groups. The first group is based on functional central limit theorems (fCLTs) in the Banach space of continuous functions endowed with the maximum metric and evaluation of the maximum of the limiting Gaussian process often using Monte-Carlo simulations with an estimated covariance structure of the limit process, cf. \citet{Bunea2011, Degras2011, Degras2017, Cao2012,Cao2014}. The second group is based on the bootstrap, among others \citet{Cuevas2006, Chernozhukov2013, Chang2017}, \citet{Belloni2018}. 

	  Although these methods are asymptotically achieving the correct covering probabilities, their small sample performance is often less impressive as for example discovered in \cite{Cuevas2006}. Figure \ref{fig:CovRateSmallSamples} shows the typical behaviour observed in our simulations of the covering rates of SCBs for the (smoothed) population mean using $N$ replications of a (smoothed) Gaussian process with an unknown non-isotropic covariance function. The general pattern is that SCBs using a non-parametric bootstrap-$t$ yield too wide confidence bands and thus have overcoverage, whereas a functional Gaussian multiplier bootstrap inspired by \cite{Chernozhukov2013} and \citet{Degras2011} asymptotic SCBs lead to undercoverage. We used here SCBs for the smoothed population mean, since the R-package \emph{SCBmeanfd} by Degras does require smoothing.

	  Contrary to most of the current methods, which assume that each functional process is observed only at discrete points of their domain, we start from the viewpoint that the whole functional processes are observed and sufficiently smooth, since often the first step in data analysis is anyhow smoothing the raw data. This implies that we construct SCBs only for the smoothed population parameter and thereby circumventing the bias problem altogether to isolate the covering aspect. In principle, however, it is possible to apply our proposed method also to discretely observed data and inference on the true (non-smoothed) mean as for example in settings described in \citet{Zhang2007, Degras2011, Zhang2016}. But for the sake of simplicity we restrict ourselves to only robustify against the choice of a smoothing parameter by introducing SCBs for scale space surfaces, which were proposed in \cite{Chaudhuri1999}.
	  
	  The main contribution of this paper is proposing the construction of SCBs for general functional parameters based on the Gaussian kinematic formula for $t$-processes (tGKF) and studying its theoretical covering properties. Moreover, as an important case study we explain and prove how all the required assumptions can be satisfied for construction of SCBs for the population mean in functional signal-plus-noise models and scale spaces. 
	  A first glance on the improvement, especially for small samples, using the tGKF is shown in Figure \ref{fig:CovRateSmallSamples}.
	  
	  A second contribution is that we compare to a studentized version of the multiplier bootstrap based on residuals, which we call \emph{multiplier-$t$} bootstrap (Mult-$t$), which also improves the small sample properties of the bootstrap SCBs considerably, if the multiplier is chosen to be Rademacher random variables. Since the scope of the paper is to analyse the theoretical properties of the use of the tGKF in construction of SCBs, we do not provide any rigorous mathematical theory about the Mult-$t$ bootstrap. Theoretical analysis of this bootstrap method, in particular the influence of the choice of multiplier, is an interesting opportunity for future research. For recent work on the performance of different multipliers in similar scenarios, see e.g., \cite{Deng2017} and \cite{Davidson2008}.
	  
	  \begin{figure}\centering
		  \includegraphics[width=1\textwidth]{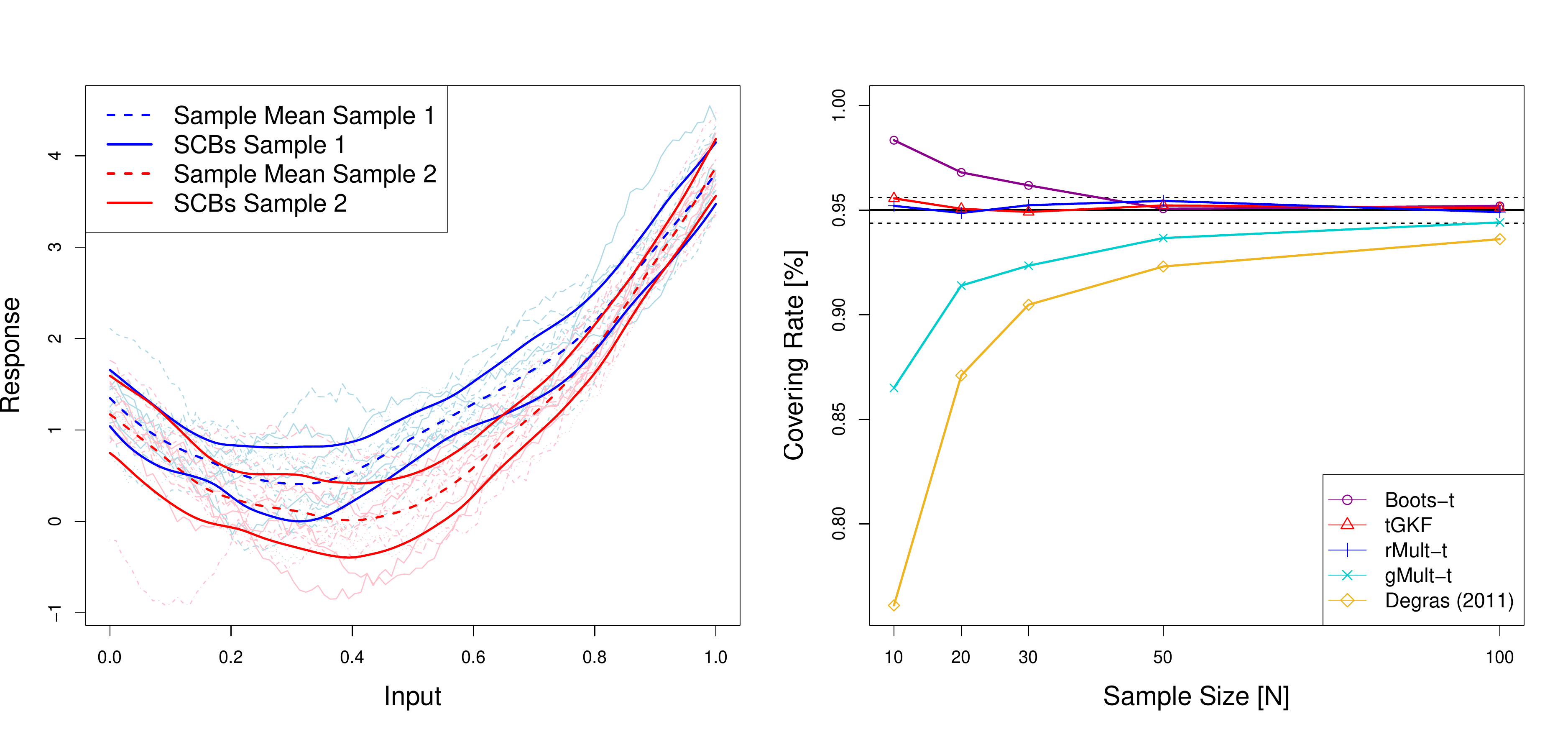}
		  \caption{ \textbf{Left:} Samples of two Gaussian process with observation noise and different population mean function and SCBs from tGKF for the smoothed data. \textbf{Right:} Example of dependence of the covering rate of SCBs of the smoothed mean function on the sample size, constructed using various methods.}
		    \label{fig:CovRateSmallSamples}
	  \end{figure}

	  SCBs usually require estimation of the quantiles of the maximum of a process. Often this process is the maximum of a process derived from a functional central limit theorem (fCLT); thus the maximum of a Gaussian process. Our proposed method tries to resolve the poor covering rates for small sample sizes by imposing two small but significant changes. Firstly, we use a pointwise $t$-statistic, since for small sample sizes and unknown variance this would be the best choice to construct confidence bands for Gaussian random variables. Secondly, we employ a formula which is known to approximate the tail distribution of the maximum of a pointwise $t$-distributed process very well. This formula is known as the \emph{Gaussian kinematic Formula for $t$-processes} (tGKF).
	  
	  In a nutshell the tGKF as proven in even more generality in \cite{Taylor2006} expresses the expected Euler characteristic (EEC) of the excursion set of a random process $F(Z_1,...,Z_N)$, $F\in \cC^2(\R^N, \R)$, derived from unit variance Gaussian processes $Z_1,...,Z_N\iid Z$ indexed by a nice compact subset of $\R^D$ in terms of a finite sum of $N$ known functions with corresponding $N$ coefficients. These coefficients are called \emph{Lipschitz-Killing curvatures} (LKCs) and depend solely on the random process $Z$. Earlier versions of this formula derived in \citet{Adler1981} for Gaussian processes were used by \cite{Worsley1996},\cite{ Worsley2004} in neuroscience for multiple testing corrections. \citet{Takemura2002} have shown that the GKF (for Gaussian processes) is closely related to the \emph{Volume of Tubes formula} dating all the way back to \cite{Working1929} and which has been applied for SCBs in nonlinear regression analysis, e.g., \cite{Johansen1990,Krivobokova2010, Lu2017}.

	  In this sense the version of the tGKF of \citet{Taylor2006} can be interpreted as a generalization of the volume of tube formula for repeated observations of functional data. The most important and valuable aspect of the tGKF is that it is a non-asymptotic formula in the sample size $N$. Hence it is suitable for applications with small sample sizes. Moreover, it has been proven in \cite{Taylor2005} that the EEC is a good approximation of the quantile of the maximum value of Gaussian- and $\chi^2$-processs over their domain. Although it is not yet proven, simulations suggest that this property seems to be also true for more complicated  Gaussian related processes like $t$-processes. Therefore, it can be used to control the family wise error rate (FWER) as in \citet{TaylorWorsley2007, Taylor2009}. To the best of our knowledge, the full potential of the tGKF for functional data seems to not have been explored yet in the statistical literature and we will try to fill this gap and tie together some loose strings for the challenge of constructing SCBs.
	  
	  Our main theoretical contributions are the following. In Theorem \ref{thm:AsymptoticValidity} we show based on the main result in \citet{Taylor2005} that asymptotically the error in the covering rate of SCBs for a function-valued population parameter based on the GKF for $t$-processes can be bounded and is small, if the targeted covering probability of the SCB is sufficiently high. This requires no Gaussianity of the observed processes. It only requires that the estimator of the targeted function-valued parameter fulfills a fCLT in the Banach space of continuous functions with a sufficiently regular Gaussian limit process. Moreover, it requires consistent estimators for the LKCs. Using this general result we derive SCBs for the population mean and the difference of population means for functional signal-plus-noise models, where we allow the error processes to be non-Gaussian. Especially we derive for such models defined over sufficiently regular domains $\cS\subset\R^D$, $D=1,2$, consistent estimators for the LKCs. These estimators are closely related to the discrete estimators in \citet{TaylorWorsley2007}, but we can even prove CLTs for our estimators. In order to deal with observation noise we give in Theorem \ref{thm:CLTscaleProcess} sufficient conditions to have weak convergence of a scale space process to a Gaussian limit extending the results from \cite{Chaudhuri2000} from regression analysis to repeated observations of functional data. Additionally, we prove that also the LKCs of this limit process can be consistently estimated and therefore Theorem \ref{thm:AsymptoticValidity} can be used to bound the error in the covering rate for SCBs of the population mean of a scale space process.
	  
	  These theoretical findings are accompanied by a simulation study using our Rpackage \emph{SCBfda}, which can be found on \url{https://github.com/ftelschow/SCBfda} and a data application to diffusion tensor imaging (DTI) fibers using a scale space approach for the difference of population means. In the simulation study we compare the performance of the tGKF approach to SCBs for different error processes mainly with bootstrap approaches and conclude that the tGKF approach does not only often give better coverings for small sample sizes, but also outperforms bootstrap approaches computationally. Moreover, the average width of the tGKF confidence bands is lower for large sample sizes.
	  
	  

	\paragraph{Organisation of the Article}
	Our article is organized as follows: In Section \ref{sec:SCB} we describe the general idea of construction of SCBs for functional parameters. Section \ref{sec:ModelAndAssumptions} defines the functional signal-plus-noise model and general properties on the error processes, which are necessary to prove most of our theorems. The next section explains how SCBs can be constructed in praxis and especially explains how the tGKF can be used for this purpose. Section \ref{sec:Asymptotics} finally proves asymptotic properties of the SCBs constructed using the tGKF for general  functional parameters. The latter require consistent estimation of the LKCs, which are discussed for the functional signal-plus-noise model in Section \ref{sec:LKCestimation} together with the new CLTs for the LKCs. Robustification using SCBs for Scale Space models can be found in \ref{sec:ObsNoise}. In Section \ref{sec:Simulations} we compare our proposed method in different simulations to competing methods to construct SCBs in the case of the functional signal-plus-noise model for various settings, which is followed in Section \ref{sec:DTI} by a data example.
	\FloatBarrier
\section{Simultaneous Confidence Bands}
    \FloatBarrier
\subsection{SCBs for Functional Parameters}\label{sec:SCB}
    We describe now a general scheme for construction of simultaneous confidence bands (SCBs) for a functional parameter $s\mapsto\eta(s)$, $s\in\cS$, where $\cS$ is a compact metric space. Note that all functions of $s$ hereafter will be assumed to belong to $\cC(\cS)$ the space of continuous functions from $\cS$ to $\R$.
    
    Assume we have estimators $s\mapsto\hat\eta_N(s)$ of $\eta$ and $s\mapsto\hat\varsigma_N(s)$ estimating a different functional parameter $s\mapsto\varsigma(s)$ fulfilling
    \begin{equation*}
	      \hspace{-0cm}\textbf{(E1)} ~ ~ ~ \tau_N\frac{\hat\eta_N-\eta}{\varsigma}\xRightarrow{N\rightarrow\infty} \cG(0,\fr)\,,~~~~~~
	      \textbf{(E2)}~~~ \Prb\left( \lim_{N\rightarrow\infty} \Vert \hat\varsigma_N - \varsigma \Vert_\infty = 0 \right) =1\,.
    \end{equation*}
    Here and whenever convergence of random continuous functions is considered, ``$\Rightarrow$'' denotes weak convergence in $\cC(\cS)$ endowed with the maximums norm $\Vert\cdot\Vert_\infty$. In \textbf{(E1)}, $\cG(0,\fr)$ is a mean zero Gaussian process with covariance function $\fr$ satisfying $\fr(s,s)=1$ for all $s\in\cS$, $\{\tau_N\}_{N\in\N}$ is a scaling sequence of positive numbers.  Assumptions \textbf{(E1)} and \textbf{(E2)} together with the functional version of Slutzky's Lemma  imply
    \begin{equation}\label{eq:RatioCLT}
		  \tau_N\frac{\hat\eta_N-\eta}{\hat\varsigma_N} \xRightarrow{N\rightarrow\infty} \cG(0,\fr)\,.
    \end{equation}
    Thus, it is easy to check that the collection of intervals 
    \begin{equation}\label{eq:SCBconstruction}
		  SCB(s,q_{\alpha,N}) = \Big[ \hat\eta_N(s) - q_{\alpha,N} \tfrac{\hat\varsigma_N(s)}{\tau_N} , \ \hat\eta_N(s) + q_{\alpha,N} \tfrac{\hat\varsigma_N(s)}{\tau_N} \Big]
    \end{equation}
    form $(1-\alpha)$-simultaneous confidence bands of $\eta$, i.e.
    \begin{equation*}
        \Prb\Big( \forall s\in \cS: \eta(s)\in SCB(s,q_{\alpha,N}) \Big) = 1-\alpha\,,
    \end{equation*}
    provided that
    \begin{equation}\label{eq:FiniteNQuantile}
		  \Prb\Bigg( \max_{s\in \cS} \tau_N\left\vert  \frac{\hat\eta_N(s) - \eta(s)}{\hat\varsigma_N(s)} \right\vert > q_{\alpha,N} \Bigg) = \alpha\,.
    \end{equation}
    Unfortunately, the quantiles $q_{\alpha,N}$ are in general unknown and need to be estimated. Therefore, the main contribution of this article is the study and comparison of different estimators for $q_{\alpha,N}$.
    
    In principle, there are two general approaches. Limit approximations try to estimate $q_{\alpha,N}$ by approximations of the quantiles of the asymptotic process, i.e.
      \begin{equation}\label{eq:AsymptoticQuantile}
		\Prb\Big( \max_{s\in \cS} \left\vert \cG(0,\fr) \right\vert > q_\alpha \Big) \geq \alpha\,,
    \end{equation}
    as done in \citet{Degras2011, Degras2017} for the special case of a signal-plus-noise model and the local linear estimator of the signal. Here usually the covariance function $\fr$ is estimated and then many samples of the Gaussian process $\cG(0,\hat\fr)$ are simulated in order to approximate $q_{\alpha}$.
    
    Better performance for finite sample sizes can be achieved by approximating $q_{\alpha,N}$ directly by bootstrap approaches such as a fully non-parametric bootstrap proposed in \citet{Degras2011}, which is inspired by the bootstrap-$t$ confidence intervals (e.g., \cite{Ciccio1996}),.
    
    Moreover, the tGKF approach, which we will introduce in Section \ref{scn:tGKFquantileEstimator}, uses a ``parametric'' estimator of $q_{\alpha,N}$, approximating the l.h.s of \eqref{eq:RatioCLT} by a $t$-process over $\cS$.        

\FloatBarrier
\subsection{Functional Signal-Plus-Noise Model}\label{sec:ModelAndAssumptions}
\FloatBarrier
	  Our leading example will be SCBs for the population mean curve in a functional signal-plus-noise model which we will introduce now. Let us assume that $\cS\subset\R^D$, $D\in \N$, is a compact set with piecewise $\cC^2$-boundary $\partial \cS$. The \emph{functional signal-plus-noise model} is given by
	  \begin{equation}\label{eq:ContModel}
		      Y(s) = \mu(s) + \sigma(s)Z(s)\,, \text{ for }s\in \cS\,.
	  \end{equation}
	  Here we assume that $\mu,\sigma$ are continously differentiable functions on $\cS$ and $Z$ is a stochastic process with zero mean and covariance function $\cov\left[ Z(s),Z(s') \right]=\mathfrak{c}(s,s')$ for $s,s'\in \cS$ satisfying $\mathfrak{c}(s,t)=1$ if and only if $s=t$. Moreover, we introduce the following useful properties of stochastic processes.
	  \begin{definition}
	      We say a stochastic process $Z$ with domain $\cS$ is \emph{$(\cL^p,\delta)-$Lipschitz}, if there is a (semi)-metric $\delta$ on $\cS$ continuous w.r.t.  the standard metric on $\R^D$ and a random variable $A$ satisfying $\EE\big[ \vert A \vert^p \big] < \infty$ such that 
		      \begin{equation}\label{eq:LipAssump}
				\big\vert Z(s)-Z(s')\big\vert \leq A \delta(s,s') \text{ for all } s,s'\in \cS
		      \end{equation}
		      and $\int_0^1 H^{1/2}(\cS,\delta,\epsilon)d\epsilon <\infty$, where $H(\cS,\delta,\epsilon)$ denotes the metric entropy function of the (semi)-metric space $(\cS,\delta)$, e.g., \citet[Def. 1.3.1.]{Adler2007}.
	  \end{definition}
	  \begin{remark}
			Any $(\cL^2,\delta)-$Lipschitz process $Z$ has necessarily almost surely continuous sample paths. Moreover, this property is the main ingredient in the version of a CLT in $\cC(\cS)$ proven in \cite{Jain1975} or \citet[Section 10.1]{Ledoux2013}. However, there are different results on CLTs in $\cC(\cS)$ with different assumptions on the process, which in principle could replace this condition. For example for $D=1$ we could also use the condition 
			\begin{equation*}
			      \EE\left[ \big(Z(s)-Z(s')\big)^2 \right] < f(\vert s-s'\vert) \text{ for }\vert s-s'\vert \text{ small and  }\int_0^1 y^{-\frac{3}{2}}f^{\frac{1}{2}}(y)\,dy<\infty
			\end{equation*}
			where $f:[0,1]\rightarrow\R_{\geq0}$ is non-decreasing near 0 and satisfies $f(0)=0$. This is due to \citet{Hahn1977}. However, $(\cL^2, \delta)$-Lipschitz seems to be the most tractable assumption for our purposes.
	  \end{remark}
	  \begin{definition}\label{def:CSmoment}
		    We say a process $Z$ has finite $p$-th $\cC(\cS)$-moment, if $\EE\big[ \Vert Z(s) \Vert^p_\infty \big]<\infty$.
	  \end{definition}
	  \begin{proposition}\label{prop:redundancy}
            Any $(\cL^p,\delta)-$Lipschitz process over a compact set $\cS$ has finite $p$-th $\cC(\cS)$-moment.
	  \end{proposition}

	  \begin{remark}\label{rem:boundedmoment}
            Since any continuous Gaussian process satisfies the finite $p$-th $\cC(\cS)$-moment condition, cf. \cite{Landau1970}, it is possible to proof a reverse of Proposition \ref{prop:redundancy} for continuously differentiable Gaussian processes. Moreover, finite $p$-th $\cC(\cS)$-moment conditions are often assumed, when uniform consistency of estimates of the covariance function are required, e.g., \citet{Hall2006} and \citet{Li2010}.
	  \end{remark}
\FloatBarrier
    
    \section{Estimation of the Quantile}\label{sec:quantiles}
    
     This section describes different estimators for the quantiles $q_{\alpha,N}$ defined by equation \eqref{eq:FiniteNQuantile}. Especially, we propose using the Gaussian kinematic formula for $t$-processes (tGKF) as proven in \citet{Taylor2006}. Moreover, we describe different bootstrap estimators. 
     
    
	  \subsection{Estimation of the Quantile Using the tGKF}\label{scn:tGKFquantileEstimator}
		\subsubsection{The Gaussian Kinematic Formula for t-processes}
		A $t_{N-1}$-process $T$ over an index set $\cS$ is a stochastic process such that $T(s)$ is $t_{N-1}$-distributed for all $s\in \cS$. In order to match the setting in \cite{Taylor2006} we assume that the process is given by
		\begin{equation}
			    T(s) =  \frac{\sqrt{N} G_{N}(s) }{ \sqrt{\frac{1}{N-1}\sum_{n=1}^{N-1} G_n(s)^2} }\,,
		\end{equation}
		where $G_1,...,G_{N}\iid \cG$ are mean zero, variance one Gaussian processes.  Let us define $D^I\!\cG=\tfrac{\partial^{\vert I \vert} \cG}{\partial s_{I_1}...\partial s_{I_K} }$, where $K=\vert I \vert$ denotes the number of elements in the multi-index $I$. Then we require the following assumptions:
		\begin{enumerate}[align=left, leftmargin=1.6cm]
			\item[\textbf{(G1)}]  $\cG$ has almost surely $\cC^2-$sample paths.
			  \item[\textbf{(G2)}] The  joint distribution of $\Big(D^d\cG(s), D^{(d,l)}\cG(s)\Big)$ is nondegenerate for all $s\in \cS$ and $d,l=1,...,D$.
			  \item[\textbf{(G3)}]There is an $\epsilon >0$ such that $$\EE\!\left[ \big( D^{(d,l)}\cG(s) -D^{(d,l)}\cG(s')\big)^2 \right] \leq K \big\vert \log\Vert s-s' \Vert \big\vert^{-(1+\gamma)} $$ for all $d,l=1,...,D$ and for all $\vert s-t\vert <\epsilon$. Here $K>0$ and $\gamma>0$ are finite constants.
		\end{enumerate}
		\begin{remark}\label{rem:GKFassumptions}
		    Assumption \textbf{(G3)} is satisfied for any process $\cG$ having almost surely $\cC^3-$sample paths and all third derivatives have finite second $\cC(\cS)$-moment, see Definiton \ref{def:CSmoment}. In particular, this holds true for any Gaussian process with $\cC^3-$sample paths. For completeness the argument is carried out in more detail in the appendix.
		\end{remark}

		Under these assumptions the tGKF is an exact, analytical formula of the expectation of the Euler characteristic $\chi$ of the excursion sets $A(T,u)=\left\{ s\in \cS ~\vert~ T(s) >u  \right\}$ of $T$. This formula as proven in  \citet{Taylor2006} or  \citet{Adler2007} is
		\begin{equation} \label{eq:EEC}
		\mathbb{E}\left[ \chi\big( A(T,u)\big)  \right] =  L_0(\cS) \rho_0^{t_{N-1}} (u) + \sum_{d=1}^D L_d(\cS,\cG) \rho_d^{t_{N-1}}(u)\,,~D\in \N\,,
		\end{equation}
		where $\rho_d^{t_{N-1}}$, $d=1,...,D$, is the $d$-th Euler characteristic density of a $t_{N-1}$-process, which can be found for example in \citet{TaylorWorsley2007}. Moreover, $ L_d$  denotes the $d$-th Lipschitz killing curvature, which only depends on $\cG$ and the parameter space $\cS$. Note that $ L_0(\cS) = \chi(\cS)$.
		
		Equation \eqref{eq:EEC} is useful, since by the expected Euler characteristic heuristic (EECH) (see \cite{Taylor2005} ), we expect
		\begin{equation}\label{eq:EECH}
			    \frac{1}{2}\Prb\Big( \max_{s\in \cS} \vert T(s) \vert > u \Big) \leq \Prb\Big( \max_{s\in \cS} T(s) > u \Big) \approx \mathbb{E}\left[ \chi(u) \right]\,,
		\end{equation}
		to be a good approximation for large thresholds $u$. In the case that $\cS\subset\R$ this approximation actually is always from above. This is due to the fact that the Euler characteristic of one-dimensional sets is always non-negative and hence using the Markov-inequality we obtain
		\begin{equation}
			    \frac{1}{2}\Prb\Big( \max_{s\in \cS} \vert T(s) \vert > u \Big) \leq \Prb\Big( \chi\big( T(s) > u \big) \geq 1 \Big) \leq \mathbb{E}\left[ \chi(u) \right]\,.
		\end{equation}
		The same argument is heuristically valid for high enough thresholds in any dimension, since the excursion set will with high probabilty consist mostly of simply-connected sets. Thus, the EC of the excursion set will be positive with high probability. Therefore, once again we expect that the EEC is an approximation from above for the excursion probability.
		
		Notably, the tGKF Equation \eqref{eq:EEC} had a predecessor for the EEC of a Gaussian processes $\cG$. Remarkably, the only difference between these formulas is that the Euler characteristic densities are different, see \citet[p.315, (12.4.2)]{Adler2007}. Moreover, the approximation error when using the EEC instead of the excursion probabilities of the Gaussian process $\cG$ has been analytically quantified in \cite{Taylor2005}.

		  \subsubsection{The tGKF-estimator of $q_{\alpha,N}$}
		 Assume that we have consistent estimators $\hat  L_d(\cS,\cG)$ for the LKCs $ L_d(\cS,\cG)$ for all $d=1,...,D$ (we provide such estimators for $D = 1,2$ in Section \ref{sec:LKCestimation}). Then a natural estimator $\hat q_{\alpha,N}$ of the quantile $q_{\alpha,N}$ is the largest solution $u$ of
		  \begin{equation}\label{eq:ualpha}
			    \widehat{\mathrm{EEC}}_{t_{N-1}}(u)= L_0(\cS) \rho_0^{t_{N-1}} (u) + \sum_{d=1}^D\hat L_d(\cS,\cG) \rho_d^{t_{N-1}}(u) = \frac{\alpha}{2}\,.
		  \end{equation}
		  The following result will be used in the next section for the proof of the accuracy of the SCBs derived using the tGKF estimator $q_{\alpha,N}$.
		  \begin{theorem}\label{thm:consistencyquantile}
			    Assume that $\hat L^N_d(\cS,\cG)$ is a consistent estimator of $ L_d(\cS,\cG)$. Then $\hat q_{\alpha,N}$ given by \eqref{eq:ualpha} converges almost surely for $N$ tending to infinity to the largest solution $\tilde q_\alpha$ of
			    \begin{equation}\label{eq:ualphaGauss}
					\mathrm{EEC}_{\cG}(u)= L_0(\cS) \rho_0^{\cG}(u) + \sum_{d=1}^D L_d(\cS,\cG) \rho_d^{\cG}(u) = \frac{\alpha}{2}\,,
			    \end{equation}
			    where $\rho_D^{\cG}$ are the Euler characteristic densities of a Gaussian process, which can be found in \citet[p.315, (12.4.2)]{Adler2007}.
		  \end{theorem}

      \subsection{Estimation of the Quantile Using the Bootstrap}\label{scn:bootsest}

      As alternatives to the approach using the tGKF we discuss a non-parametric bootstrap-$t$ and a multiplier bootstrap estimator of the quantile $q_{\alpha,N}$ defined in Equation \eqref{eq:FiniteNQuantile}.

	    \subsubsection{Non-parametric bootstrap-$t$} Based on \citet{Degras2011} we review a bootstrap estimator of the quantile $q_{\alpha,N}$.  Assume that the estimators $s\mapsto\hat\eta_N(s)$ and $s\mapsto\hat\varsigma_N(s)$ are obtained from a sample $Y_1,...,Y_N\iid Y$ of random functions, then the \emph{non-parametric bootstrap-$t$} estimator of $q_{\alpha,N}$ is obtained as follows
	    \begin{enumerate}
			\item Resample from $Y_1,...,Y_N$ with replacement to produce a bootstrap sample $Y_{1}^*,...,Y_{N}^*$.
			\item Compute $\hat\eta^*_N$ and $\hat\varsigma_N^*$ using the sample $Y_{1}^*,...,Y_{N}^*$.
			\item Compute $T^* = \max_{s\in \cS} \tau_N \big\vert \hat\eta^*_N(s) - \hat\eta_N(s) \big\vert / \hat\varsigma_N^*(s) $.
			\item Repeat steps 1 to 3 many times to approximate the conditional law $\mathcal{L}^*=\mathcal{L}\big( T^*\, \vert~ Y_1,...,Y_N \big)$ and take the $(1-\alpha)\cdot100\%$ quantile of $\mathcal{L}^*$ to estimate  $q_{\alpha,N}$.
	    \end{enumerate}
	    \begin{remark}
	     Note that the variance in the denominator is also bootstrapped, which corresponds to the standard bootstrap-$t$ approach for confidence intervals, cf. \cite{Ciccio1996}. This is done in order to mimic the l.h.s. in \eqref{eq:RatioCLT}, and improves the small sample coverage. Results of this bootstrap variant will be discussed in our simulations.
	    \end{remark}

	    We expect that this estimator works well for large enough sample sizes. Although \cite{Degras2011} introduced it especially for small sample sizes, there is not much hope that it will perform well in this case, since it is well known that confidence bands for a finite dimensional parameter based on the bootstrap-$t$ have highly variable end points, cf., \citet[Section 3.3.3]{Good2005}. Evidence that this remains the case in the functional world will be given in our simulations in Section \ref{sec:Simulations}.
	    
	    \subsubsection{Multiplier-$t$ Bootstrap}
	    The second bootstrap method, which we introduce, builds on residuals and a version of the multiplier (or Wild) bootstrap (e.g., \citet{Mammen1993}) designed for the maximum of sums of $N$ independent random variables in high dimensions as discussed in detail by \citet{Chernozhukov2013}. Briefly, the approach in \citet{Chernozhukov2013} is as follows. Let $X_1,...,X_N$ be independent random vectors in $\R^K$, $N,K\in\N$ with $\EE\big[X_n\big]=0$ and finite covariance $\EE\big[X_nX_n^T\big]$ for all $n\in\{1,...,N\}$. Define $X_{n}^T=\big(X_{n1},...,X_{nK}\big)$ and assume there are $c,C\in\R_{>0}$ such that $c<\EE\big[ X_{nk}^2 \big]<C$ for all $n\in\{1,...,N\}$ and all $k\in\{1,...,K\}$. Under these assumptions it is shown in \citet[Theorem 3.1]{Chernozhukov2013} that the quantiles of the distribution of
	    \begin{equation*}
		    \max_{k\in\{1,...,K\}} \frac{1}{\sqrt{N}} \sum_{n=1}^N X_{nk}
	    \end{equation*}
	    can be asymptotically consistently estimated by the quantiles of the multiplier bootstrap i.e., by the distribution of
	    \begin{equation*}
		      \max_{k\in\{1,...,K\}} \frac{1}{\sqrt{N}} \sum_{n=1}^N g_nX_{nk}
	    \end{equation*}
	    with multipliers $g_1,...,g_N\iid\cN(0,1)$ given the data $X_1,...,X_N$, even if $K\gg N$.
	    
	    We adapt this approach to the estimation of $q_{\alpha,N}$ for functional parameters. Therefore we again assume that the estimators $s\mapsto\hat\eta_N(s)$ and $s\mapsto\hat\varsigma_N(s)$ are obtained from an i.i.d. sample $Y_1,...,Y_N$ of random functions. However, here we also need to assume that we have residuals $R_n^N$ for $n=1,...,N$ satisfying
	    $
            R_n^N/\varsigma \xRightarrow{N\rightarrow\infty} \cG(0,\fr)
	    $
        and mutual independence for $N$ tending to infinity. For example for the signal-plus-noise model, where $\hat\eta_N=\hat\mu_N$ and $\hat\varsigma_N=\hat\sigma_N$ are the pointwise sample mean and the pointwise sample standard deviation, the residuals  $R_n^N=\sqrt{\tfrac{N}{N-1}}(Y_n-\hat\mu_N)$ do satisfy these conditions, if the error process $Z$ is $(\cL^2,\delta)$-Lipschitz and has finite second $\cC(\cS)$-moment.
        Thus, using these assumptions we obtain that
	    \begin{equation*}
		    \max_{ s \in \cS}  \sqrt{N}\left\vert  \frac{\hat\eta_N(s) - \eta(s)}{\hat\varsigma_N(s)} \right\vert ~~~~\text{ and }~~~~ \max_{ s \in \cS} \frac{1}{\sqrt{N}} \left\vert \sum_{n=1}^N \tfrac{R_n^N\!(s)}{\hat\varsigma_N(s)} \right\vert
	    \end{equation*}
	    have approximately the same distribution. Algorithmitically, our proposed multiplier bootstrap is
	    \begin{enumerate}
		      \item Compute residuals $R_1^N,...,R_N^N$ and multipliers $g_1,...,g_N\overset{\text{i.i.d.}}{\sim}g$ with $\EE[g]=0$ and $\var[g]=1$
		      \item Estimate $\hat\varsigma_N^{*}(s)$ from $g_1Y_1(s) ,...,g_NY_N(s)$.
		      \item Compute $T^*(s) = \frac{1}{\sqrt{N}} \sum_{n=1}^N g_n\frac{R_n^N\!(s)}{\hat\varsigma_N^*(s)}$.
		      \item Repeat steps 1 to 3 many times to approximate the conditional law $\mathcal{L}^*=\mathcal{L}\big( T^*\, \vert~ Y_1,...,Y_N \big)$ and take the $(1-\alpha)\cdot100\%$ quantile of $\mathcal{L}^*$ to estimate  $q_{\alpha,N}$.
	    \end{enumerate}
	    In our simulations we use Gaussian multipliers as proposed in \cite{Chernozhukov2013}, but find that Rademacher multipliers as used in regression models with heteroskedastic noise, e.g. \cite{Davidson2008}, perform much better for small sample sizes.
	    \begin{remark}
		    Again the choice to compute bootstrapped versions of $\hat\varsigma_N$ mimics the true distribution of the maximum of the random process on the l.h.s of \eqref{eq:RatioCLT} for finite $N$ better than just applying the multipliers to $R_n^N/\hat\varsigma_N$.
	    \end{remark}

      \section{Asymptotic Covering Rates}\label{sec:Asymptotics}
        \subsection{Asymptotic SCBs for Functional Parameters}
     This section discusses the accuracy of the SCBs derived using the tGKF. Since the expected Euler characteristic of the excursion sets are only approximations of the excursion probabilities, there is no hope to prove that the covering of these confidence bands is exact. Especially, if $\alpha$ is large the approximation fails badly and will usually lead to confidence bands that are too wide. However, for values of $\alpha<0.1$ typically used for confidence bands, the EEC approximation works astonishingly well. Theoretically, this has been made precise for Gaussian processes in Theorem 4.3 of \cite{Taylor2005}, which is the main ingredient in the proof of the next result. Additionally, it relies on the fCLT \eqref{eq:RatioCLT} and the consistency of $q_{\alpha,N}$ for $q_\alpha$ proved in Theorem \ref{thm:consistencyquantile}.
      \begin{theorem}\label{thm:AsymptoticValidity}
		Assume \textbf{(E1-2)} and assume that the limiting Gaussian process $\mathcal{G}(0, \fr)$ satisfies \textbf{(G1-3)}. Moreover, let $\hat  L_{d}^N$ be a sequence of consistent estimators of $ L_{d}$ for $d=1,...,D$. Then there exists an $\alpha'\in(0,1)$ such that for all $\alpha\leq\alpha'$ we have that the SCBs defined in Equation \eqref{eq:SCBconstruction} fullfill
		\begin{equation*}
			\lim_{N\rightarrow\infty}\Big\vert 1-\alpha-\Prb\big( ~\forall s\in \cS:~\eta(s)\in SCB(s,\hat q_{\alpha,N} ) \big) \Big\vert \leq e^{-\big(\tfrac{1}{2}+\tfrac{1}{2\sigma_c^2}\big) \tilde q_\alpha^2 } < e^{- \tfrac{\tilde q_\alpha^2}{2} }\,,
		  \end{equation*}
		  where $\sigma_c^2$ is the critical variance of an associated process of $\mathcal{G}(0, \fr)$, $\hat q_{\alpha,N}$ is the quantile estimated using equation \eqref{eq:ualpha} and $\tilde q_\alpha$ is defined in Theorem \ref{thm:consistencyquantile}.
      \end{theorem}
      Typically, in our simulations we have that, for $\alpha=0.05$,  the quantile $\tilde q_\alpha$ is about $3$ leading to an upper bound of $\approx0.011$, if we use the weaker bound without the critical variance.

	\subsection{Asymptotic SCBs for the Signal-Plus-Noise Model}
	As an immediate application of Theorem \ref{thm:AsymptoticValidity} we derive now SCBs for the population mean and the difference of population means in one and two sample scenarios of the signal-plus-noise model introduced in Section \ref{sec:ModelAndAssumptions}. The derivation of consistent estimators for the LKCs will be postponed to the next section.
      
	\begin{theorem}[Asymptotic SCBs for Signal-Plus-Noise Model]\label{thm:GenericAsymptoticCBs}
		  Let $Y_1,...,Y_N\iid Y$ be a sample of model \eqref{eq:ContModel} and assume $Z$ is an $(\cL^2,\delta)-$Lipschitz process.  Define $\textbf{Y}(s)=\big( Y_1(s),...,Y_N(s) \big)$.
		  \begin{enumerate}
			\item[(i)] Then the estimators
			\begin{equation}
					  \hat\mu_N(s) = \overline{\textbf{Y}(s)} = \frac{1}{N}\sum_{n=1}^N Y(s)\,, ~~ \hat\sigma_N^2(s) = \widehat{\rm var}_N\!\left[\!\,\textbf{Y}(s)\,\right] = \frac{1}{N-1} \sum_{n=1}^N \big( Y_n(s) - \bar Y(s) \big)^2\,,
			\end{equation}
				    fullfill the conditions \textbf{(E1-2)} with $\tau_N=\sqrt{N}$, $\eta=\mu$, $\varsigma=\sigma$ and $\fr=\fc$.
			 \item[(ii)] If there are consistent estimators of the LKCs $ L_d$ and $Z$ has $\cC^3$-sample paths and all partial derivatives up to order $3$ of $Z$ are $(\cL^2,\delta)-$Lipschitz processes with finite $\cC(\cS)$-variances and $\cG(0,\fc)$ fullfills the non-degeneracy condition \textbf{(G2)}, then the accuracy result of Theorem \ref{thm:AsymptoticValidity} holds true for the SCBs
			  \begin{equation*}
			    SCB(s,\hat q_{\alpha,N} ) = \hat\mu_N(s) \pm \hat q_{\alpha,N}\frac{\hat\sigma_N(s)}{\sqrt{N}}
			  \end{equation*}
			  with $\hat q_{\alpha,N}$ of Theorem \ref{thm:consistencyquantile}.
		  \end{enumerate}
	\end{theorem}
	\begin{remark}
	    A simple condition on $Z$ to ensure that  $\cG(0,\fc)$ fulfills the non-degeneracy condition \textbf{(G2)} is that for all $d,l\in\{1,...,D\}$ we have that $\cov\!\left[ \left( D^lZ(s), D^{(d,l)} Z(s)\right) \right]$ has full rank for all $s\in\cS$. A proof is provided in Lemma \ref{lemma:G2assumption} in the appendix.
	\end{remark}

	\begin{theorem}[Asymptotic SCBs for Difference of Means of Two Signal-Plus-Noise Models]\label{thm:GenericAsymptoticCBdiffs}
		  Let $Y_1,...,Y_N\iid Y$ and $X_1,...,X_M\iid X$ be independent samples, where
		  \begin{equation}
		      Y(s) = \mu_Y(s) + \sigma_{Y}(s) Z_Y(s) ~ ~ ~ \text{ and } ~ ~ ~ X(s) = \mu_X(s) + \sigma_{X}(s) Z_X(s)\,,
		  \end{equation}
		  with $Z_Y, Z_X$ both $(\cL^2,\delta)-$Lipschitz processes and assume that $c=\lim_{N,M\rightarrow\infty}N/M$. Then
		  \begin{enumerate}
		   \item[(i)] Condition \textbf{(E1)} is satisfied, i.e. 
				\begin{equation*}
					    \frac{ \sqrt{N+M-2} \big( \overline{\textbf{Y}} -  \overline{\textbf{X}} - \mu_Y + \mu_X \big) }{ \sqrt{ (1+c^{-1})\hat\sigma_N(\textbf{Y})^2 + (1+c)\hat\sigma_N(\textbf{X})^2 }} \xRightarrow{N,M\rightarrow\infty}  \mathcal{G} = \frac{ \sqrt{1+c^{-1}}\sigma_{Y} \cG_Y -  \sqrt{1+c}\sigma_{X}\cG_X}{\sqrt{(1+c^{-1})\sigma_{Y}^2 + (1+c)\sigma_{X}^2}}\,,
				\end{equation*}
				where $G_Y, G_X$ are Gaussian processes with the same covariance structures as $X$ and $Y$ and the denominator converges uniformly almost surely, i.e. condition \textbf{(E2)} is satisfied.
				\item[(ii)] If there are consistent estimators of the LKCs $L_d$ of $~\cG$ and $Z_X, Z_Y$ have $\cC^3$-sample paths, fullfill the non-degeneracy condition \textbf{(G2)} and all their partial derivatives are $(\cL^2,\delta)-$Lipschitz processes with finite $\cC(\cS)$-variances, then the accuracy result of Theorem \ref{thm:AsymptoticValidity} holds true for the SCBs
			  \begin{equation*}
			    SCB(s,\hat q_{\alpha,N} ) = \hat\mu_N(s) \pm \hat q_{\alpha,N}\frac{ \sqrt{(1+c^{-1})\sigma_{Y}^2 + (1+c)\sigma_{X}^2} }{\sqrt{N+M}}
			  \end{equation*}
			  with $\hat q_{\alpha,N}$ of Theorem \ref{thm:consistencyquantile}.
		  \end{enumerate}
	\end{theorem}
	
	\section{Estimation of the LKCs for Signal-Plus-Noise Models}\label{sec:LKCestimation}
	We turn now to consistent estimators of the LKCs. However, we treat only the case where $Z$ is a process defined over a domain $\cS$, where $\cS$ is assumed to be either a compact collection of intervals in $\R$ or a compact two-dimensional domain of $\R^2$ with a piecewise $\cC^2$-boundary $\partial \cS$. We restrict ourselves to this setting, since there are simple, explicit formulas for the LKCs.
	
		
	\subsection{Definition of the LKCs for $D=1$ and $D=2$}       
       The basic idea behind the Lipschitz-Killing curvatures $ L_d(\cS,Z)$ is that they are the intrinsic volumes of $\cS$ with respect to the pseudo-metric $\tau(s,s') = \sqrt{ \var\left[ Z(s) - Z(s') \right] }\,,$ which is induced by the pseudo Riemannian metric given in standard coordinates of $\R^D$, cf. \citet{Taylor2003},
	\begin{equation}\label{eq:covDefinition}
		    \Lambda_{dl}(Z,s)=\Lambda_{dl}(s) = \cov\left[D^d\!Z(s), D^l\!Z(s) \right]\,,~~~d,l=1,...,D\,.
	\end{equation}
	For the cases $D=1,2$ the general expressions from \citet{Adler2007} can be nicely simplified. For $D=1$, i.e. $\cS\subset\R$, the only unknown is $ L_1$, which is given by
	\begin{equation}\label{eq:LKC1d}
		   L_1(\cS,Z) = {\rm vol}_1(\cS, \Lambda) = \int_S \sqrt{ \var\left[ \tfrac{dZ}{ds}(s) \right]} ds\,.
	\end{equation}
	In the case of $\cS\subset\R^2$ with piecewise $\cC^2$-boundary $\partial \cS$ parametrized by the piecewise $\cC^2$-function $\gamma:[0,1]\rightarrow\R^2$ the LKCs are given by
	\begin{align}\label{eq:LKC2d}
	\begin{split}
		     L_1 &= \frac{1}{2} {\rm length}(\partial \cS, \Lambda) = \frac{1}{2}\int_0^1 \sqrt{ \tfrac{d\gamma}{dt}^T\!\!(t)\Lambda\big(\gamma(t)\big)\tfrac{d\gamma}{dt}(t)}\,dt\\
		     L_2 &=  {\rm vol}_2(\cS, \Lambda) = \int_\cS \sqrt{ \det\big( \Lambda(s)\big) } ds_1ds_2\,.
		    \end{split}
	\end{align}
	Note that by the invariance of the length of a curve the particular choice of the parametrization $\gamma$ of $\cS$ does not change the value of $ L_1$. 
	
	\subsection{ Simple Estimators Based on Residuals}\label{Sec:SimpleLKCestimator}
	In order to estimate the unknown LKCs $ L_1$ and $ L_2$ from a sample $Y_1,...,Y_N\sim Y$ of model \eqref{eq:ContModel}, we assume that the estimators $\hat\mu_N$ of $\mu$ and $\hat\sigma_N$ of $\sigma$ satisfy $\textbf{(E1-2)}$ with $\fr=\fc$ and
	\begin{enumerate}[align=left, leftmargin=1.6cm]
		      \item[\textbf{(L)} ] $\lim_{N\rightarrow\infty} \Vert D^d\!\left(\sigma/\hat\sigma_N\right) \Vert_\infty = 0$ for all $d=1,...,D$ almost surely.
	\end{enumerate}
	These conditions imply that the normed residuals
	\begin{equation}\label{def:normedresiduals}
		    R_n(s) = \big( Y_n(s) - \hat\mu(s)\big) / \hat\sigma_N(s) \,,~~~ s \in \cS\,,n=1,...,N
	\end{equation}
	and its gradient converge uniformly almost surely to the random processes $Z_n$ and $\nabla Z_n$. Let us denote the vector of residuals by $\textbf{R}(s)=\big(R_1(s),...,R_N(s)\big)$. In view of equation \eqref{eq:LKC1d} it is hence natural to estimate the LKC $ L_1$ for $D=1$ by
	\begin{equation}\label{eq:L1IntegralEstimator}
		    \hat{L}_1^N = \int_0^1 \sqrt{  \widehat{\rm var} \left[ \tfrac{d}{ds}\! \left(\tfrac{\textbf{R}(s)}{\hat\sigma_N(s)}\right) \right] } \,ds\,,
	\end{equation}
	and using equations \eqref{eq:LKC2d} the LKCs for $D=2$ by
	\begin{align}\label{eq:L1L2IntegralEstimator}
	      \begin{split}
		     \hat L_1^N &= \frac{1}{2} {\rm length}(\partial \cS, \Lambda) = \frac{1}{2}\int_0^1 \sqrt{ \tfrac{d\gamma}{dt}^T\!\!(t)\hat \Lambda_N\big(\gamma(t)\big)\tfrac{d\gamma}{dt}(t)}\,dt\\
		     \hat L_2^N &=  {\rm vol}_2(\cS, \Lambda) = \int_\cS \sqrt{ \det\big( \hat \Lambda_N(s)\big) } ds_1ds_2\,,
	       \end{split}
	\end{align}
	where $\hat\Lambda_N(s) = \widehat{\var}_N\left[ \nabla\textbf{R}(s) \right]$ is the empirical covariance matrix of $\nabla\textbf{R}(s)$. Thus, we need to study mainly the properties of this estimator to ensure properties of the estimated LKCs.
	\begin{theorem}\label{thm:CovarianceConsistency}
		Assume $Z$ and $D^d Z$ for $d=1,..., D$ are $(\cL^2,\delta)-$Lipschitz processes and $\textbf{(E2)}$ and $\textbf{(L)}$ hold true. Then, we obtain
		\begin{equation*}
				  \lim_{N\rightarrow\infty} \left\Vert  \hat\Lambda_N - \Lambda \right\Vert_\infty = 0 \text{ almost surely}\,.
		\end{equation*}
		  If $Z$ and $D^d Z$ for $d=1,..., D$ are even $(\cL^4,\delta)-$Lipschitz processes, we obtain
		\begin{equation*}
		    \sqrt{N}\Big( \iota\big( \hat\Lambda_N \big) - \iota\big( \Lambda\big) \Big) \xRightarrow{N\rightarrow\infty}  \cG\big(0, \mathfrak{t} \big)
		\end{equation*}
		with $\iota: {\rm Sym(2)}\rightarrow \R^3$ mapping
		\begin{equation*}
		    \begin{pmatrix}
				      a & b \\
				      b & c 
		         \end{pmatrix} \mapsto (a,b,c)
		\end{equation*}
		and the matrix valued covariance function $\mathfrak{t}:\cS\times\cS\rightarrow {\rm Sym(3)}$ is given componentwise by
		\begin{equation}\label{eq:asymCovariance}
		  \mathfrak{t}_{dl}(s,s') = \cov\left[ D^d Z(s)D^l Z(s), D^d Z(s')D^l Z(s') \right]
		\end{equation}
		for all $s,s'\in\cS$ and $d,l=1,...,D$.
	\end{theorem}
	
	From this theorem we can draw two interesting results for $D<3$. The first one is a consistency results of the LKCs. In a sense this is just reproving the result for the discrete estimator given in \citet{TaylorWorsley2007} without the need to refer to convergences of meshes and making underlying assumptions more visible. Especially, it becomes obvious that the Gaussianity assumption on $Z$ is not required, since we only have to estimate the covariance matrix of the gradient consistently.
	\begin{theorem}[Consistency of LKCs]\label{thm:LKCConsistency}
		    Under the setting of Theorem \ref{thm:CovarianceConsistency} it follows for $d=1,2$ that
		    \begin{equation*}
		      \Prb\left(  \lim_{N\rightarrow\infty} \hat L_d^N = L _d \right) = 1\,.
		    \end{equation*}
	\end{theorem}
	An immediate corollary can be drawn for the generic pointwise estimators in the functional signal-plus-noise Model.
	\begin{corollary}\label{cor:LKCConsistency}
		      Let $\hat\sigma^2_N(s) = \widehat{\var}_N \left[ \textbf{Y}(s) \right]$ and assume that $Z$ and $D^d\! Z$ for $d=1,2$ are $(\cL^2,\delta)-$Lipschitz processes. Then for $d=1,2$  we obtain
		      \begin{equation}
				  \Prb\left(  \lim_{n\rightarrow\infty} \hat L_d^N = L_d \right) = 1\,.
		      \end{equation}
	\end{corollary}
	\begin{remark}
		  If the process $Z$ is Gaussian then we can even infer that the estimators \eqref{eq:L1IntegralEstimator} and \eqref{eq:L1L2IntegralEstimator} are unbiased, if $\hat\sigma^2_N(s) = \widehat{\var}_N \left[ \textbf{Y}(s) \right]$, see \citet{TaylorWorsley2007}.
	\end{remark}

	The second conclusion we can draw from this approach is that we can derive a CLT for the estimator of the LKCs.
	\begin{theorem}[CLT for LKCs]\label{thm:CLTlkc}
		    Assume $Z$ and $D^d\!Z$ for $d=1,2$ are even $(\cL^4,\delta)-$Lipschitz processes and $\textbf{(E2)}$ and $\textbf{(L)}$ hold true. Then, if $\cS\subset\R$,
		    \begin{equation*}
			    \sqrt{N} \left( \hat L^N_1 -  L_1 \right) \rightarrow \frac{1}{2}\int_{\cS} \frac{G(s)}{\sqrt{\Lambda(s)}} ds\,,
		    \end{equation*}
		    and, if $\cS\subset\R^2$, we obtain
		        \begin{align*}
			    \sqrt{N} \Big( \big(\hat L^N_1, \hat L^N_2 \big) - \big(  L_1,  L_2\big) \Big) \rightarrow   &\Bigg( \frac{1}{2}\int_0^1 \frac{1}{\sqrt{\gamma'(s)^T\Lambda\big(\gamma(s)\big)\gamma'(s)}} {\rm tr}\Bigg( \Lambda\big(\gamma(s)\big) \iota\Big( \cG\big(\gamma(s)\big) \Big) \Bigg) ds,\\ &~~~~~~~~~~~~~~~\int_{\cS} \frac{1}{\sqrt{\det\big(\Lambda(s)\big)}} {\rm tr}\Big( \Lambda(s) \iota\big( {\rm diag}(1,-1,1)\cG(s) \big) \Big) ds \Bigg)\,.
		      \end{align*}
		  where $\cG(s)$ is a Gaussian process with mean zero and covariance function $\mathfrak{t}$ as in Theorem \ref{thm:CovarianceConsistency} and $\gamma$ is a parametrization of the boundary $\partial \cS$.
	 \end{theorem}
	 \begin{corollary}\label{cor:1DCLT}
		      Assume additionally to the Assumptions of Theorem \ref{thm:CLTlkc} that $Z$ is Gaussian with covariance function $\fc$ and $\cS\subset\R$. Then, we have the simplified representation
		      \begin{equation}
			\sqrt{N}\left( \hat  L_1^N -  L_1 \right)\xRightarrow{N\rightarrow\infty}\cN\big(0, \tau^2 \big)\,,
		      \end{equation}
		      where
		      \begin{equation*}
			\tau^2 = \frac{1}{2} \int_S \int_S \frac{\dot{\mathfrak{c}}(s,s')^2 }{ \sqrt{\dot{\mathfrak{c}}(s,s')\dot{\mathfrak{c}}(s,s')} }  ds ds'~ ~ \text{with}~ ~\dot{\mathfrak{c}}(s,s')= D^1_{s}\!D^1_{s'} \mathfrak{c}(s,s')
		      \end{equation*}
	  \end{corollary}
	  
	  \subsection{Estimation of LKCs in the Two Sample Problem}
	  In the previous section we dealt with the case of estimating the LKCs in the case of one sample. This idea can be extended to the two sample case leading to a consistent estimator of the LKCs of the asymptotic process given in Theorem \ref{thm:GenericAsymptoticCBdiffs}. Using the same notations as before we have that the difference in mean satisfies the following fCLT
	  \begin{equation*}
		  \frac{ \sqrt{N+M-2} \big( \bar{\textbf{Y}} -  \bar{\textbf{X}} -\mu_Y + \mu_X \big) }{ \sqrt{ (1+c^{-1})\hat\sigma^2_N(\textbf{Y}) + (1+c)\hat\sigma^2_N(\textbf{X}) }} \xRightarrow{N,M\rightarrow\infty}  \cG = \frac{ \sqrt{1+c^{-1}}\sigma_{Y}\cG_Y -  \sqrt{1+c}\sigma_{X}\cG_X}{\sqrt{(1+c^{-1})\sigma_{Y}^2 + (1+c)\sigma_{X}^2}}\,.
	  \end{equation*}
	  Therefore, note that independence of the samples implies that the covariance matrix defined in equation \eqref{eq:covDefinition} splits into a sum of covariance matrices depeding on $\cG_X$ and $\cG_Y$, i.e.,
	  \begin{equation*}
	      \Lambda(\cG) =  \Lambda\!\left( \frac{ \sqrt{1+c^{-1}}\sigma_{Y}\cG_Y}{\sqrt{(1+c^{-1})\sigma_{Y}^2 + (1+c)\sigma_{X}^2}} \right) +   \Lambda\!\left( \frac{ \sqrt{1+c}\sigma_{X}\cG_X}{\sqrt{(1+c^{-1})\sigma_{Y}^2 + (1+c)\sigma_{X}^2}} \right)\,.
	  \end{equation*}
The latter summands, however, under the assumptions of Theorem \ref{thm:CovarianceConsistency} can be separately consistently estimated using the previously discussed method and the residuals
	  \begin{align*}
		R^Y_n   &=  \frac{ \sqrt{1+c^{-1}} \cdot (Y_n -  \bar{\textbf{Y}}) }{ \sqrt{(1+c^{-1})\hat\sigma_N(\textbf{Y}) +  (1+c)\hat\sigma_N(\textbf{X})}}\\
		R^X_n  &=  \frac{ \sqrt{1+c} \cdot (X_n -  \bar{\textbf{X}}) }{ \sqrt{(1+c^{-1})\hat\sigma^2_N(\textbf{Y}) +  (1+c)\hat\sigma^2_N(\textbf{X})}}\,.
	  \end{align*}
	 The sum of these estimators is a consistent estimator of $\Lambda(\cG)$ and thus the estimator described in the previous section based on the estimate of $\Lambda(\cG)$ using the above residuals gives a consistent estimator of the LKCs of $\cG$.

	\section{Discrete Measurements, Observation Noise and Scale Spaces}\label{sec:ObsNoise}
	
	We discuss now the realistic scenario, where the data is observed on a measurement grid with possible additive measurement noise. Again for simplicity we restrict ourselves to the signal-plus-noise model, which is
	\begin{equation}\label{eq:DiscreteModel}
		  Y(s_{p}) = \mu(s_{p}) + \sigma(s_{p})Z(s_{p}) + \varepsilon(s_{p})\,, \text{ for }p=1,...,P\,,
	\end{equation}
	where $\mathbf{S} = (s_1,...,s_p)\in \cS\subset\R^{D\times P}$. Here we assume that $\varepsilon$ is stochastic process on $\cS$ with finite second $\cC(\cS)$-moment representing the observation noise and covariance function $\mathfrak{e}$. We also assume that $\varepsilon$, $Z$ and $\mathbf{S}$ are mutually independent. We are interested in doing inference on the mean $\mu$, if we observe a sample $(\mathbf{S}_{1}, Y_1) ,..., (\mathbf{S}_{N}, Y_N) \iid (\mathbf{S}, Y)$.

	Since it is still a challenge to optimally choose the bandwidth for finite and especially small sample sizes, if we want to estimate the population mean directly, we rather rely here on the idea of scale spaces, which was originally introduced for regression in \cite{Chaudhuri1999} and \cite{Chaudhuri2000}. The goal is then to show that we can provide SCBs for the population mean simultaneously across different scales. Therefore we define now scale spaces derived from Pristley-Rao smoothing estimators which could, however, be replaced by local polynomial or other appropriate linear smoothers.
	\begin{definition}[Scale Space Process]
		  We define the \emph{Scale Space Process} with respect to a continuous kernel functions $K: \tilde\cS \times [h_0,h_1] \rightarrow \R $ with $\tilde\cS\supset\cS$ and $0<h_0<h_1<\infty$ corresponding to Model \eqref{eq:DiscreteModel} as
		  \begin{equation*}
			    \tilde Y(s,h) =  \frac{1}{P}\sum_{p=1}^P Y(s_p) K( s - s_p,h )
		  \end{equation*}
		  with scale mean
		  \begin{equation*}
			   \tilde\mu(s,h) = \frac{1}{P}\sum_{p=1}^P \mu(s_p) K( s - s_p,h )\,.
		  \end{equation*}
	\end{definition}
	In order to apply the previously presented theory we have to obtain first a functional CLT. The version we present is similar to Theorem 3.2 of \citet{Chaudhuri2000} with the difference that we consider the limit with respect to the number of observed curves and include the case of possibly having the (random) measurement points depend on the number of samples, too. The regression version in \citet{Chaudhuri2000} only treats the limit of observed measurement points going to infinity.
	\begin{theorem}\label{thm:CLTscaleProcess}
		  Let $Y_1,...,Y_N \iid Y$ be a sample from Model \eqref{eq:DiscreteModel}. Assume further that $Z$ has finite second $\cC(\cS)$-moment, $\mathbf{S}$ is a possibly random $P$ vector, where $P$ is allowed to depend on $N$. Moreover, assume that $\max_{s\in\cS} \sigma(s) \leq B <\infty$ and
		  \begin{align}\label{assumption:covarianceExists}
		  	\begin{split}
			   \mathfrak{r}\big( (s,h),(s',h') \big) = \lim_{ N\rightarrow\infty } \frac{1}{P^2} \sum_{p,p'=1}^{P} &\EE\left[ \big( \sigma(s_{p})\sigma(s_{p'})\fc(s_p,s_{p'}) +  \mathfrak{e}(s_p,s_{p'}) \big) K( s - s_p, h) K( s' - s_{p'}, h')\right]
			\end{split}
		  \end{align}
		  exists for all $(s,h), (s',h')\in\cS\times\cH$, where the expectations are w.r.t. $\mathbf{S}$. Finally, assume that the smoothing kernel $(s,h)\mapsto K( s ,h )$ is $\alpha$-H\"older continuous. Then in $\cC(\cS\times\cH)$
		  \begin{equation*}
			    \sqrt{N}\left( N^{-1}\sum_{n=1}^N \tilde Y_n(s,h) - \EE\left[\tilde\mu(s,h)\right] \right)\xRightarrow{N\rightarrow\infty} \cG(0,\mathfrak{r})\,.
		  \end{equation*}
	\end{theorem}

	\begin{remark}\label{remark:nonRandS}
		  Note that an important example covered by the above Theorem is the case where $\mathbf{S}\subset[0,1]$ is non-random and given by $s_p=(p-0.5)/P$ for $p=1,...,P$ for all $n=1,...,N$ and $\varepsilon(s_{1}),...,\varepsilon(s_{P})\iid \cN(0,\eta^2)$. If $P$ is independent of $N$, the Assumption \eqref{assumption:covarianceExists} is trivially satisfied. If $P\rightarrow\infty$ as $N\rightarrow\infty$ it is suffcient that the following integral exists and is finite for all $(s,h),(s',h')\in \cS\times\cH$
		  \begin{align*}
			\mathfrak{r}\big( (s,h),(s',h') \big) = \int_\cS \int_\cS \big(\sigma(\tau)\sigma(\tau')\fc(\tau,\tau') + \mathfrak{e}(\tau,\tau')\big)K(s-\tau,h)K(s'-\tau',h')\,d\tau d\tau'\,,
		  \end{align*}
		  in order to have Assumption \eqref{assumption:covarianceExists} satisfied, which follows for example if $\fc$ and  $\mathfrak{e}$ are continuous.
	\end{remark}

	In order to use Theorem \ref{thm:AsymptoticValidity} we only have to show that we can estimate the LKCs consistently and the assumptions of the GKF are satisfied. The next Proposition gives sufficient conditions on the kernel $K$ to ensure this.
	\begin{proposition}\label{prop:ScaleSCBsufficientCond}
	    Assume the setting of Theorem \ref{thm:CLTscaleProcess} and additionally assume that the kernel $K\in\cC^3(\cS\times\cH)$. Define
	    \begin{equation}
		    \hat{\tilde\mu}(s,h) = \frac{1}{N} \sum_{n=1}^N \tilde Y_n(s,h)~ ~\text{ and }~ ~ \hat{\tilde\sigma}(s,h) =  \sqrt{\frac{1}{N} \sum_{n=1}^N \left(\tilde Y_n(s,h)-\hat{\tilde\mu}(s,h)\right)^2}
	    \end{equation}
	    and assume that $\mathfrak{r}$ has continuous partial derivatives up to order $3$ and the covariance matrices of
	    \begin{equation*}
			  \lim_{ N\rightarrow\infty } \frac{1}{P} \sum_{p=1}^{P} \big( \sigma(s_{p}) Z(s_p) + \varepsilon(s_{p}) \big) \left( \frac{\partial K( s - s_p, h)}{\partial x}, \frac{\partial^2 K( s - s_p, h)}{\partial s\partial h} \right)\,, ~~x = s \text{ or }h\,,
	    \end{equation*}
	   have rank $2$. Then $\hat{\tilde\mu}(s,h)$ and $\hat{\tilde\sigma}(s,h)$ satisfy Assumptions \textbf{(L)}, \textbf{(E2)} and \textbf{(G3)}. Thus, all assumptions of Theorem \ref{thm:AsymptoticValidity} are satisfied.
	\end{proposition}
	\begin{remark}
	 Assume the situation of Remark \ref{remark:nonRandS}. Then the assumption that $\mathfrak{r}$ has continuous partial derivatives up to order $3$ follows directly from the assumption $K\in\cC^3(\cS\times\cH)$.
	\end{remark}

    \section{Simulations}\label{sec:Simulations}
    \FloatBarrier
    We conduct simulations for evaluation of the performance of different methods for constructing SCBs for the population mean of a signal-plus-noise model. We show that the tGKF is a reliable, fast competitor to bootstrap methods, performing well even for small sample sizes. Additionally, the average width of the tGKF SCBs has less variance in our simulations than bootstrap SCBs. On the other hand, we show that among the bootstrap versions our proposed multiplier-$t$ bootstrap with Rademacher weights performs best for symmetric distributions.
    
    Simulations in this section are always based on $5,000$ Monte Carlo simulations in order to estimate the covering rate of the SCBs. We mostly compare SCBs constructed using the tGKF, the non-parametric bootstrap-$t$ (Boots-t) and the multiplier-t with Gaussian (gMult-t) or Rademacher (rMult-t) multipliers. We always use $5,000$ bootstrap replicates.
    
    \subsection{Coverage: Smooth Gaussian Case}\label{subsec:smoothgausssimulation}
    This first set of simulations deals with the most favourable case for the tGKF method. We simulate samples from the following smooth signal-plus-noise models \eqref{eq:ContModel} (call them Model $A$, $B$, $C$)
    \begin{align*}
	      \mathbf{Model~A:}~~ Y^A(s) &= \sin(8\pi s)  \exp(-3s) +  \frac{(0.6-s)^2+1}{6}\cdot \frac{ \mathbf{a}^T\mathbf{K}^A(s)}{\Vert \mathbf{K}^A(s) \Vert} \\
	      \mathbf{Model~B:}~~Y^B(s) &= \sin(8\pi s)  \exp(-3s) +  \frac{(0.6-s)^2+1}{6}\cdot  \frac{ \mathbf{b}^T \mathbf{K}^B(s)}{\Vert \mathbf{K}^B(s) \Vert} \\
	      \mathbf{Model~C:}~~Y^C(s) &= s_1s_2 + \frac{s_1+1}{s_2^2+1}\cdot \frac{ \mathbf{c}^T\mathbf{K}^C(s)}{\Vert \mathbf{K}^C(s) \Vert}\,,~~~s=(s_1,s_2)\in[0,1]^2
    \end{align*}
    with $\mathbf{a}\sim\cN(0, I_{7\times7})$, $\mathbf{b}\sim\cN(0, I_{21\times21})$ and $\mathbf{c}\sim\cN(0, I_{36\times36})$.  Moreover, the vector $\mathbf{K}^A(s)$ has entries $K^A_i(s) = \binom{6}{i}s^i(1-s)^{6-i}$ the $(i,6)$-th Bernstein polynomial for $i=0,...,6$, $\mathbf{K}^B(s)$ has entries $K^C_i(s) = \exp\big(-\frac{\left(s-x_i\right)^2}{2h_i^2}\big)$ with $x_i=i/21$, $h_i=0.04$ for $ i<10$, $h_{11}=0.2$ and $h_i=0.08$ for $ i>10$ and $\mathbf{K}^C(s)$ is the vector of all entries from the $6\times6$-matrix $K_{ij}(s) = \exp\big(-\frac{ \Vert s-x_{ij}\Vert^2 }{ 2h^2 }\big)$ with $x_{ij}= (i, j)/6$ with $h=0.06$.
    
  Examples of sample paths of the signal-plus-noise models and the error processes, as well as the simulation results are shown in Figure \ref{fig:ResultsSimulationSmoothGaussian}. We simulated samples from Model $A$ and $B$ on an equidistant grid of $200$ points of $[0,1]$. Model $C$ was simulated on an equidistant grid with $50$ points in each dimension.
  
  The simulations suggest that the tGKF method and the rMult-t outperforms all competitors, since even for small sample sizes they achieve the correct covering rates.
  \begin{figure}\centering
	\includegraphics[width=1\textwidth]{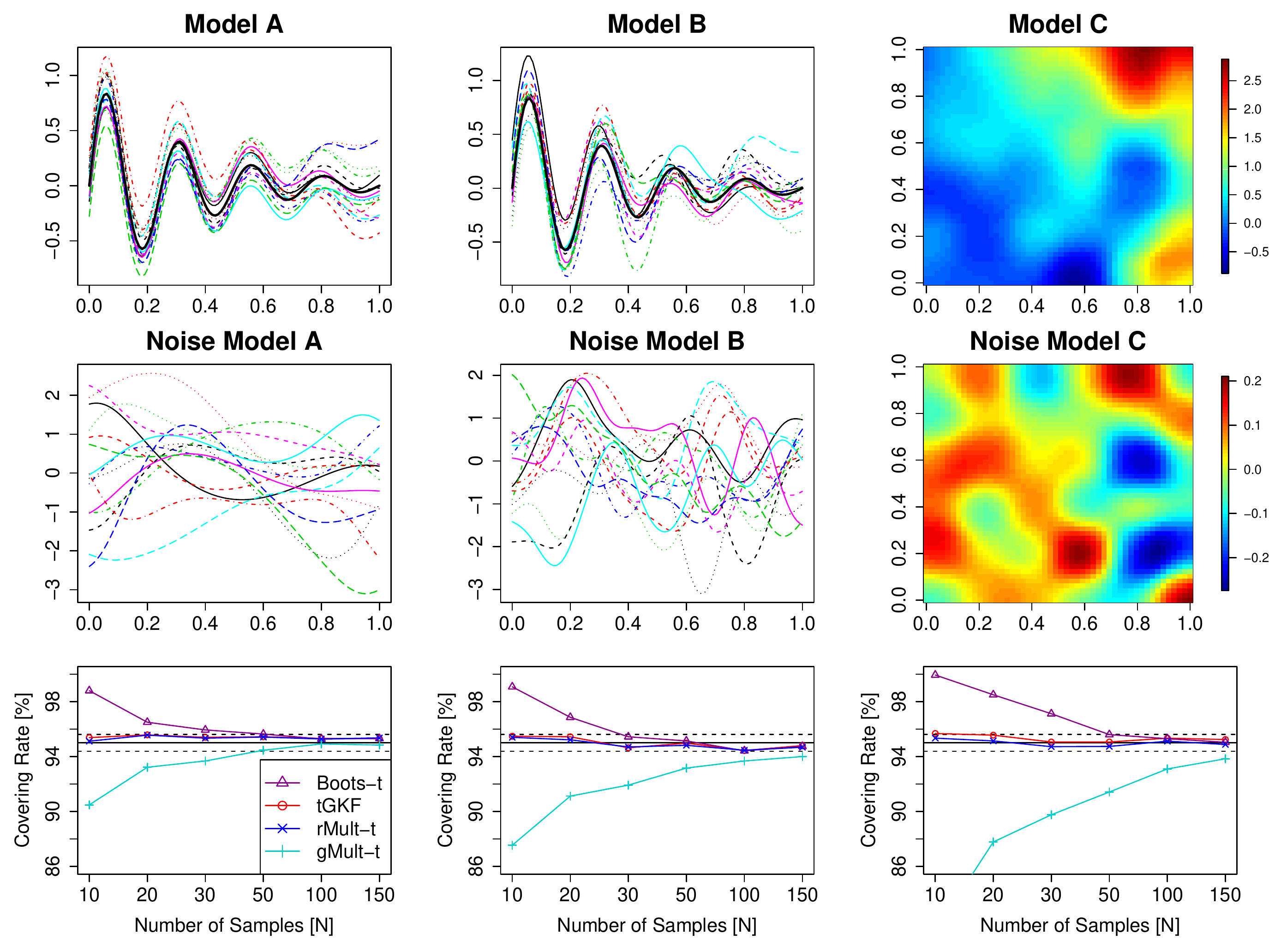}
	\caption{ Simulation result for smooth Gaussian processes. \textbf{Top row:} samples from the signal-plus-noise models. \textbf{Middle row:} samples from the error processes. \textbf{Bottom row:} simulated covering rates. The solid black line is the targeted level of the SCBs and the dashed black line are twice the standard error for a Bernoulli random variable with $p=0.95$.}
	  \label{fig:ResultsSimulationSmoothGaussian}
  \end{figure}
  \FloatBarrier
    \subsection{Coverage: Smooth Non-Gaussian Case}
    \FloatBarrier
    The tGKF method is based on a formula valid for Gaussian processes only. However, we have seen in Section \ref{sec:Asymptotics} that under appropriate conditions the covering is expected to be good asymptotically even if this condition is not fulfilled. For the simulations here we use Model $A$ with the only change that $\mathbf{a}$ has i.i.d. entries of a Student's $t_3/\sqrt{3}$ random variable, which is non-Gaussian, but still symmetric. A second simulation study tackles the question of whether non-Gaussian but close to Gaussian processes still will have reasonable covering. Here we use Model $B$ where $\mathbf{b}$ has i.i.d. entries of $(\chi^2_\nu -\nu)/\sqrt{2\nu}$ random variables for different parameters of $\nu$. These random variables are non-symmetric, but for $\nu\rightarrow\infty$ they converge to a standard normal.
      \begin{figure}\centering
	\includegraphics[width=1\textwidth]{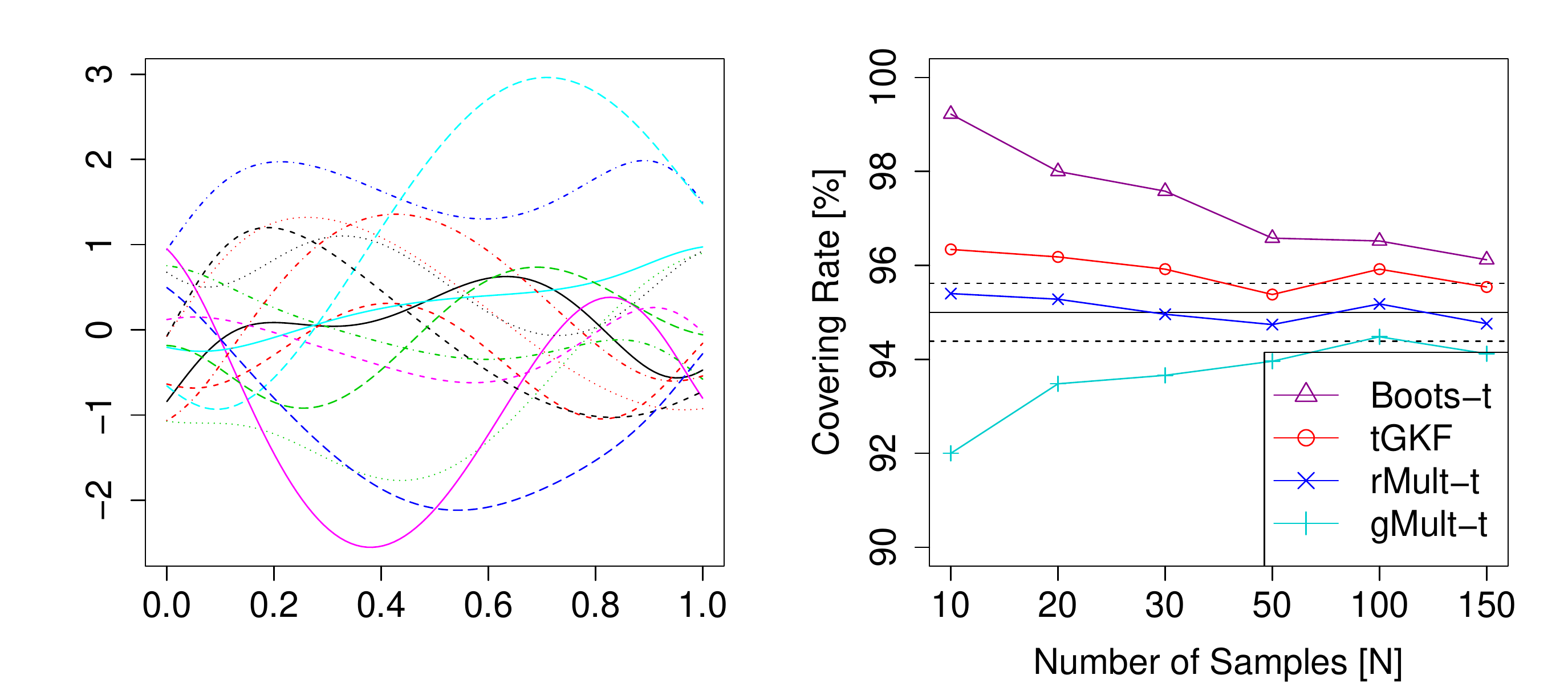}
	\caption{ Simulation result for smooth non Gaussian processes (Model A). \textbf{Left:} samples from the error processes. \textbf{Right:} simulated covering rates. The solid black line is the targeted level of the SCBs and the dashed black line are twice the standard error for a Bernoulli random variable with $p=0.95$.}
	  \label{fig:ResultsSimulationSmoothNonGaussian}
  \end{figure}
      \begin{figure}\centering
	\includegraphics[width=1\textwidth]{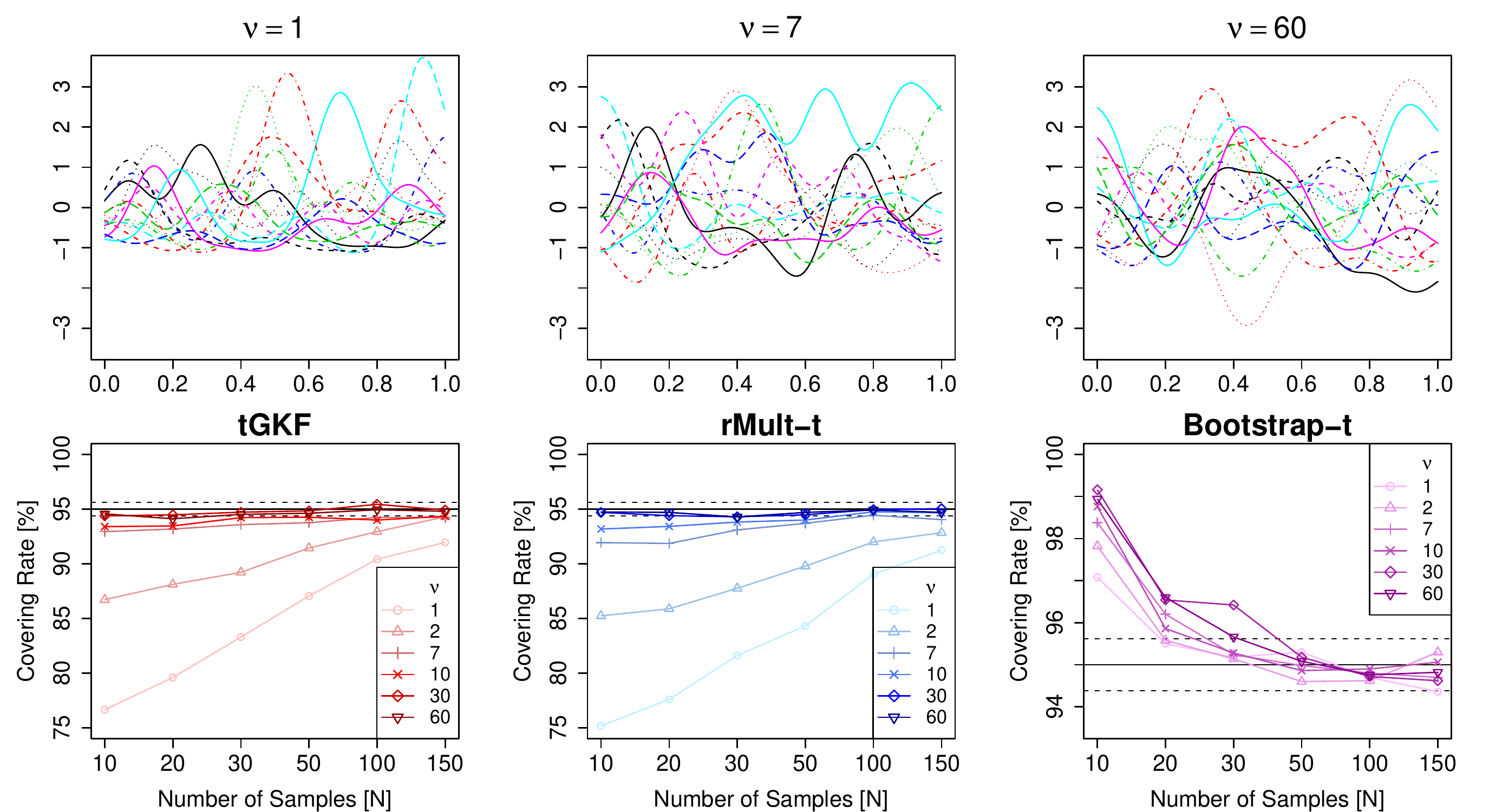}
	\caption{ Simulation result for smooth non Gaussian processes (Model B). \textbf{Left:} samples from the error processes. \textbf{Right:} simulated covering rates. The solid black line is the targeted level of the SCBs and the dashed black line are twice the standard error for a Bernoulli random variable with $p=0.95$.}
	  \label{fig:ResultsSimulationSmoothNonGaussian}
  \end{figure}
  
    Figure \ref{fig:ResultsSimulationSmoothNonGaussian} shows that in the symmetric distribution case of Model $A$ the tGKF has some over coverage for small sample sizes, but eventually it is in the targeted region for $95\%$ coverage. We observed that covering is usually quite close to the targeted covering for symmetric non-Gaussian distributions. The case of non-symmetric distributions, however, produces usually undercovering for the tGKF. However, as predicted by Theorem \ref{thm:GenericAsymptoticCBs} eventually for large $N$ it gets close to the targeted covering. In the case of non-symmetric distributions the bootstrap-t seems to perform better since it converges faster to the correct covering rate and under the symmetry condition the rMult-$t$ still works exceptionally good. This is assumed to be the case since it preserves all moments up to the fourthed.

    \FloatBarrier  
    \subsection{Average Width and Variance of Different SCBs}
    \FloatBarrier
    A second important feature of the performance of confidence bands is their average width and the variance of the width, since it is preferable to have the smallest possible width that still has the correct coverage. Moreover, the width of a SCB should remain stable meaning that its variance should be small. For the SCBs discussed in this article the width can be defined as twice the estimated quantile $\hat q_{N,\alpha}$ of the corresponding method. Thus, we provide simulations of quantile estimates for various methods for the Gaussian Model $B$ and the non-Gaussian Model $B$ with $\nu=7$. The optimal quantile $q_\alpha$ is simulated using a Monte-Carlo simulation with $50,000$ replications.
    
    We compare Degras' asymptotic method from the R-package \emph{SCBmeanfd}, the non-parametric bootstrap-$t$, multiplier-$t$, non parametric bootstrap and a simple multiplier bootstrap. Here the latter two methods use the variance of the original sample instead of its bootstrapped version, which we described in Section \ref{scn:bootsest}.
    
    Table \ref{table:GaussWidth} and \ref{table:NonGaussWidth} show the simulation results. We can draw two main conclusions from these tables. First, the tGKF method has the smallest standard error among the compared methods in the widths while still having good coverage (at least asymptotically in the non-Gaussian case). Second, the main competitor --the bootstrap-t-- has a much higher standard error in the width, but its coverage converges faster in the highly non-Gaussian and asymmetric case.

\begin{table}[ht]\footnotesize
\centering
\begin{tabular}{|l|cccccc|}
  \hline
 \textbf{N}& $ \mathbf{ 10 }$ & $ \mathbf{ 20 }$ & $ \mathbf{ 30 }$ & $ \mathbf{ 50 }$ & $ \mathbf{ 100 }$ & $ \mathbf{ 150 }$ \\ 
\textbf{true} & 4.118 & 3.382 & 3.211 & 3.081 & 2.993 & 2.96 \\   \hline
  \textbf{tGKF} & $ 3.980 \pm 0.033 $ & $ 3.368 \pm 0.011 $ & $ 3.204 \pm 0.006 $ & $ 3.084 \pm 0.004 $ & $ 3.000 \pm 0.002 $ & $ 2.973 \pm 0.001 $ \\ 
  \textbf{GKF} & $ 2.929 \pm 0.015 $ & $ 2.925 \pm 0.008 $ & $ 2.923 \pm 0.005 $ & $ 2.922 \pm 0.003 $ & $ 2.921 \pm 0.002 $ & $ 2.921 \pm 0.001 $ \\ 
  \textbf{Degras} & $ 2.866 \pm 0.021 $ & $ 2.881 \pm 0.010 $ & $ 2.887 \pm 0.007 $ & $ 2.890 \pm 0.005 $ & $ 2.892 \pm 0.003 $ & $ 2.896 \pm 0.003 $ \\ 
  \textbf{Boots} & $ 2.690 \pm 0.022 $ & $ 2.813 \pm 0.013 $ & $ 2.851 \pm 0.010 $ & $ 2.878 \pm 0.007 $ & $ 2.897 \pm 0.005 $ & $ 2.903 \pm 0.004 $ \\ 
  \textbf{Boots-t} & $ 5.728 \pm 0.507 $ & $ 3.609 \pm 0.050 $ & $ 3.292 \pm 0.020 $ & $ 3.109 \pm 0.009 $ & $ 2.998 \pm 0.005 $ & $ 2.968 \pm 0.004 $ \\ 
  \textbf{gMult} & $ 3.020 \pm 0.161 $ & $ 2.921 \pm 0.032 $ & $ 2.914 \pm 0.017 $ & $ 2.910 \pm 0.008 $ & $ 2.913 \pm 0.005 $ & $ 2.913 \pm 0.004 $ \\ 
  \textbf{gMult-t} & $ 3.428 \pm 0.046 $ & $ 3.091 \pm 0.017 $ & $ 3.011 \pm 0.011 $ & $ 2.960 \pm 0.007 $ & $ 2.931 \pm 0.004 $ & $ 2.925 \pm 0.004 $ \\ 
  \textbf{rMult} & $ 2.690 \pm 0.179 $ & $ 2.766 \pm 0.032 $ & $ 2.808 \pm 0.015 $ & $ 2.849 \pm 0.009 $ & $ 2.882 \pm 0.005 $ & $ 2.893 \pm 0.004 $ \\ 
  \textbf{rMult-t} & $ 4.100 \pm 0.122 $ & $ 3.382 \pm 0.023 $ & $ 3.202 \pm 0.013 $ & $ 3.078 \pm 0.008 $ & $ 2.993 \pm 0.005 $ & $ 2.965 \pm 0.004 $ \\ 
   \hline
\end{tabular}
\caption{Average of estimates $\hat q_{\alpha,N}$ of the width and twice its standard error for different methods of SCBs and the Gaussian Model $B$. The simulations are based on 1,000 Monte Carlo simulations.}
\label{table:GaussWidth}
\end{table}

\begin{table}[ht]\footnotesize
\centering
\begin{tabular}{|l|cccccc|}
  \hline
 \textbf{N}& $ \mathbf{ 10 }$ & $ \mathbf{ 20 }$ & $ \mathbf{ 30 }$ & $ \mathbf{ 50 }$ & $ \mathbf{ 100 }$ & $ \mathbf{ 150 }$ \\ 
\textbf{true} & 4.532 & 3.628 & 3.373 & 3.19 & 3.048 & 3.004 \\  \hline
  \textbf{tGKF} & $ 3.973 \pm 0.040 $ & $ 3.365 \pm 0.013 $ & $ 3.203 \pm 0.008 $ & $ 3.083 \pm 0.004 $ & $ 3.000 \pm 0.002 $ & $ 2.973 \pm 0.001 $ \\ 
  \textbf{GKF} & $ 2.926 \pm 0.019 $ & $ 2.922 \pm 0.009 $ & $ 2.922 \pm 0.006 $ & $ 2.921 \pm 0.004 $ & $ 2.921 \pm 0.002 $ & $ 2.921 \pm 0.001 $ \\ 
  \textbf{Degras} & $ 2.864 \pm 0.021 $ & $ 2.880 \pm 0.010 $ & $ 2.888 \pm 0.007 $ & $ 2.890 \pm 0.005 $ & $ 2.893 \pm 0.003 $ & $ 2.896 \pm 0.003 $ \\ 
  \textbf{Boots} & $ 2.698 \pm 0.024 $ & $ 2.822 \pm 0.014 $ & $ 2.858 \pm 0.010 $ & $ 2.882 \pm 0.007 $ & $ 2.900 \pm 0.005 $ & $ 2.905 \pm 0.004 $ \\ 
  \textbf{Boots-t} & $ 6.165 \pm 0.688 $ & $ 3.832 \pm 0.092 $ & $ 3.456 \pm 0.040 $ & $ 3.208 \pm 0.016 $ & $ 3.053 \pm 0.006 $ & $ 3.007 \pm 0.004 $ \\ 
  \textbf{gMult} & $ 3.034 \pm 0.194 $ & $ 2.931 \pm 0.036 $ & $ 2.918 \pm 0.019 $ & $ 2.912 \pm 0.009 $ & $ 2.914 \pm 0.005 $ & $ 2.913 \pm 0.004 $ \\ 
  \textbf{gMult-t} & $ 3.432 \pm 0.045 $ & $ 3.088 \pm 0.018 $ & $ 3.005 \pm 0.011 $ & $ 2.951 \pm 0.007 $ & $ 2.922 \pm 0.005 $ & $ 2.915 \pm 0.004 $ \\ 
  \textbf{rMult} & $ 2.769 \pm 0.234 $ & $ 2.758 \pm 0.039 $ & $ 2.798 \pm 0.018 $ & $ 2.838 \pm 0.009 $ & $ 2.873 \pm 0.005 $ & $ 2.888 \pm 0.004 $ \\ 
  \textbf{rMult-t} & $ 4.110 \pm 0.133 $ & $ 3.350 \pm 0.027 $ & $ 3.179 \pm 0.015 $ & $ 3.060 \pm 0.009 $ & $ 2.984 \pm 0.005 $ & $ 2.960 \pm 0.004 $ \\ 
   \hline
\end{tabular}
\caption{Average of estimates $\hat q_{\alpha,N}$ of the width and twice its standard error for different methods of SCBs and the non-Gaussian Model $B$ with $\nu=7$. The simulations are based on 1,000 Monte Carlo simulations.}
\label{table:NonGaussWidth}
\end{table}      
    
  \FloatBarrier
  \subsection{The Influence of Observation Noise}
    \FloatBarrier
  In order to study the influence of observation noise, we evaluate the dependence of the covering rate of the smoothed mean of a signal-plus-noise model with added i.i.d. Gaussian observation noise on the bandwidth and the standard deviation of the observation noise.
    \begin{figure}\centering
		\includegraphics[width=1\textwidth]{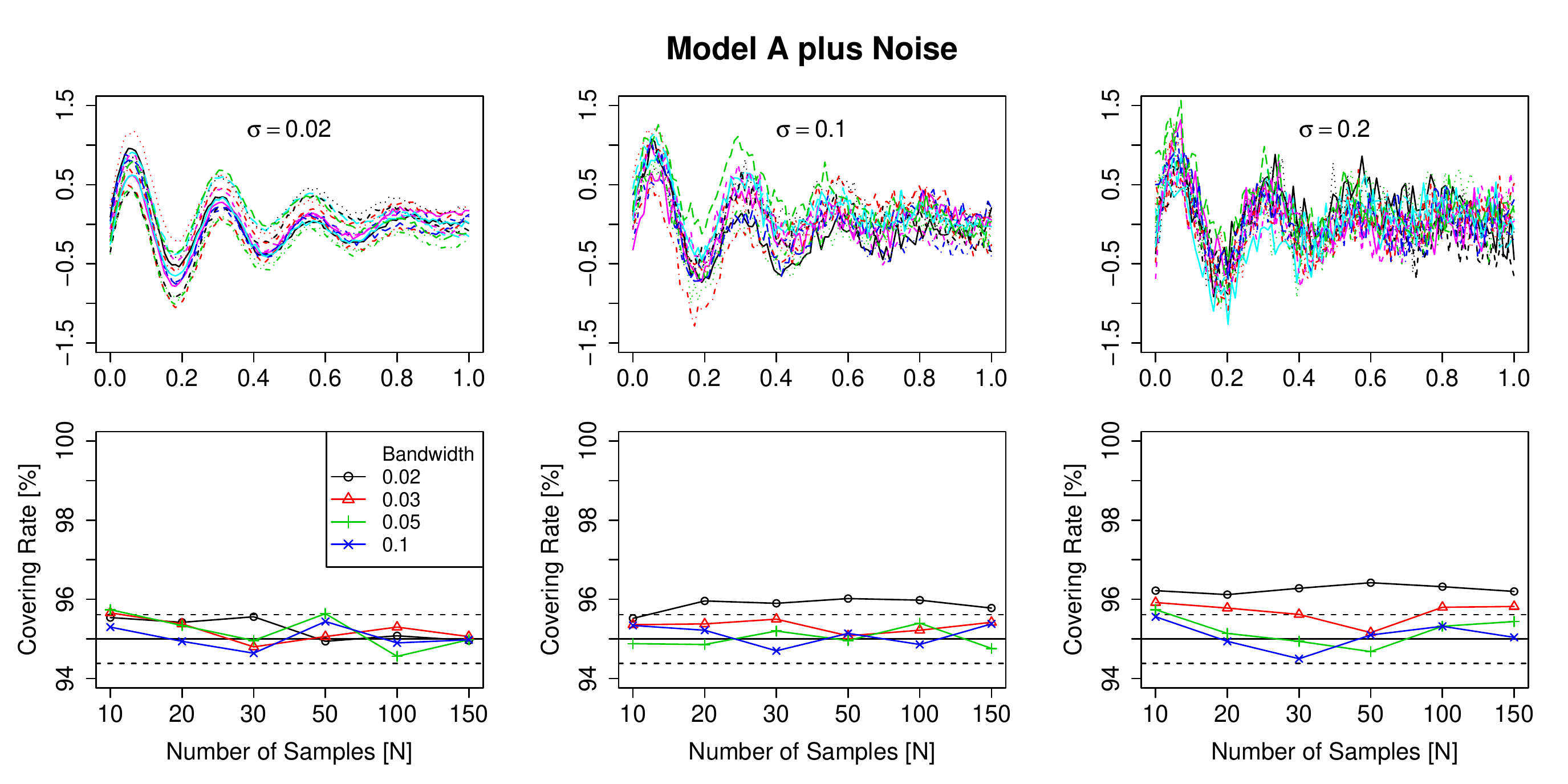}
	\caption{ Simulation results for Gaussian processes (Model A) with observation noise. \textbf{Top row:} samples from the error processes. \textbf{Bottom row:} simulated covering rates. The solid black line is the targeted level of the SCBs and the dashed black line are twice the standard error for a Bernoulli random variable with $p=0.95$.}
	  \label{fig:ResultsSimulationObsGaussianA}
  \end{figure}
  \begin{figure}\centering
		\includegraphics[width=1\textwidth]{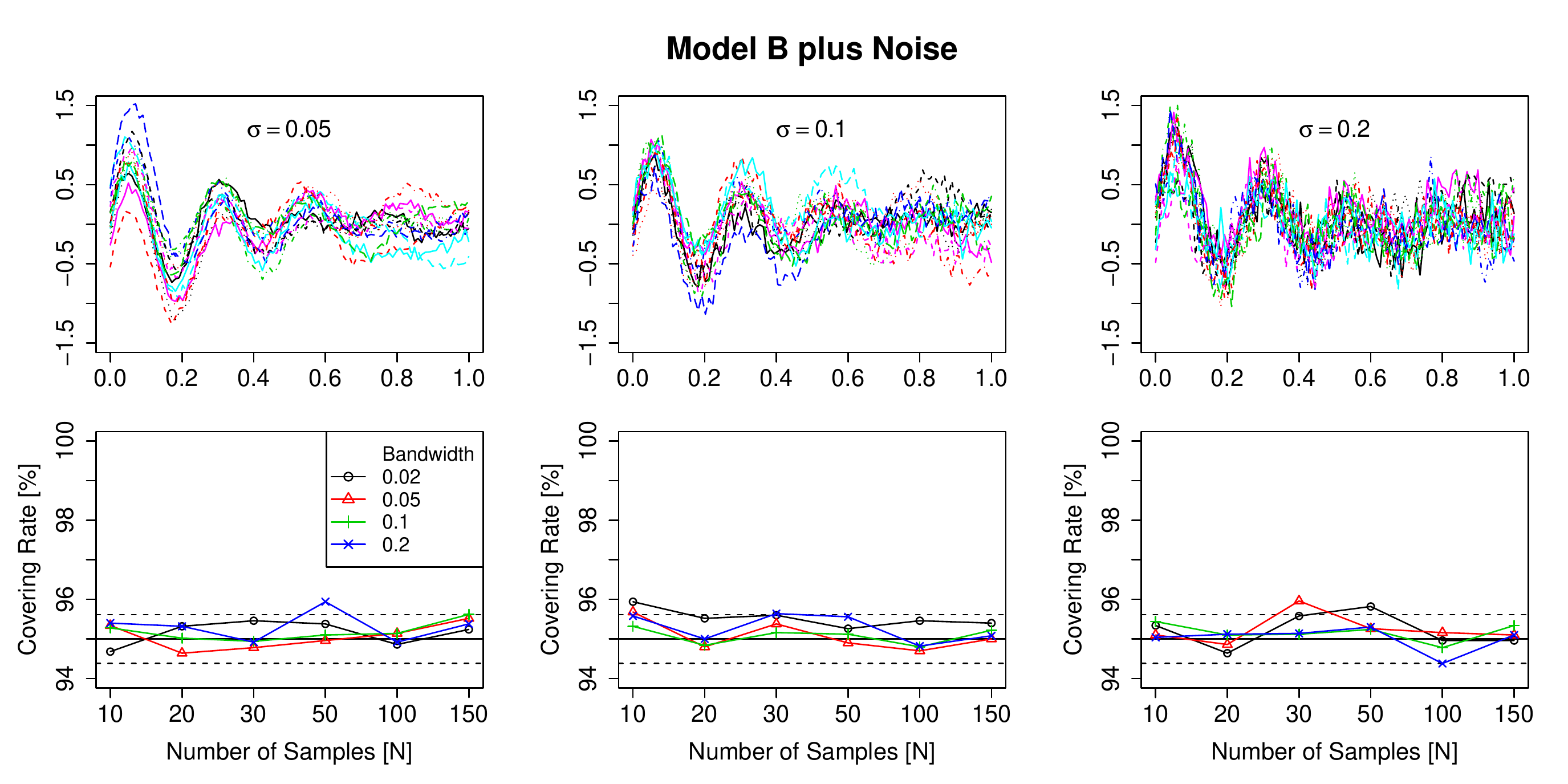}
	\caption{ Simulation results for Gaussian processes (Model B) with observation noise. \textbf{Top row:} samples from the error processes. \textbf{Bottom row:} simulated covering rates. The solid black line is the targeted level of the SCBs and the dashed black line are twice the standard error for a Bernoulli random variable with $p=0.95$.}
	  \label{fig:ResultsSimulationObsGaussianB}
  \end{figure}  
    \begin{figure}\centering
	\includegraphics[width=1\textwidth]{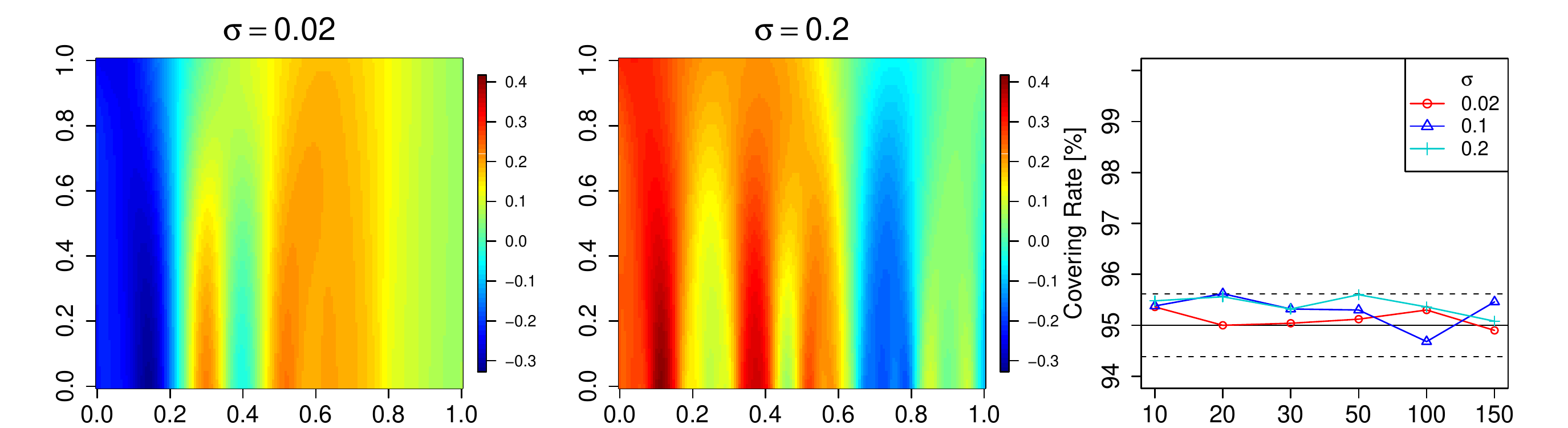}
	\caption{ Simulation results for the Scale process from the Gaussian processes of Model B with added observation noise. \textbf{Left two panels:} samples from the error processes. \textbf{Right panel:} simulated covering rates. The solid black line is the targeted level of the SCBs and the dashed black line are twice the standard error for a Bernoulli random variable with $p=0.95$.}
	  \label{fig:ResultsSimulationObsGaussianScaleB}
  \end{figure}
  For the simulations we generate samples from the Gaussian Model A and Model B on an equidistant grid of $[0,1]$ with 100 points and add $\cN(0,\sigma)$-distributed independent observation noise with $\sigma\in\{0.02, 0.1, 0.2\}$. Afterwards we smooth the samples with a Gaussian kernel with bandwidths $h\in\{0.02,0.03, 0.05,0.1\}$. The smoothed curves are evaluated on a equidistant grid with 400 points. The results of these simulations are shown in Figure \ref{fig:ResultsSimulationObsGaussianA} and \ref{fig:ResultsSimulationObsGaussianB}.
  
  We also study the covering rate of the population mean of the scale process of Model B. The generation of the samples is the same as for the previous simulations. The only difference is that instead of smoothing with one bandwidth we construct the scale space surface. Here we use a equidistant grid of 100 points of the interval $[0.02,0.1]$. The results can be found in Figure \ref{fig:ResultsSimulationObsGaussianScaleB}.
  
  \FloatBarrier
  \subsection{SCBs for the Difference of Population Means of Two Independent Samples}
  \FloatBarrier
  Since we previously studied the single sample case in great detail, we will only present the case of one dimensional smooth Gaussian noise processes in the two sample scenario. Moreover, we only report the results for the tGKF approach. The previous observations regarding the other methods, however, carry over to this case.
  
 The simulations are designed as follows. We generate two samples of sizes $N$ and $M$ such that $N/M=c\in\{1,2,4\}$. We are interested in four different scenarios. The first scenario is the most favourable having the same correlation structure and the same variance function. Here we use for both samples the Gaussian Model $A$ from subsection \ref{subsec:smoothgausssimulation}. In all remaining scenarios one of the samples will always be this error process. In order to check the effect of the two samples having different correlation structures, we use Gaussian Model $B$ as the second sample from subsection \ref{subsec:smoothgausssimulation}. For dependence on the variance, while the correlation structure is the same, we change the variance function in the Gaussian Model $A$ to $\sigma^2(s)=0.04$ for the second sample. As an example process where both the correlation and the variance are different we use Gaussian Model $B$ with the modification that the error process has pointwise variance $\sigma^2(s)=0.04$. The results of these simulations are shown in Figure \ref{fig:2sampleSCBresults}
    \begin{figure}\centering
	\includegraphics[width=1\textwidth]{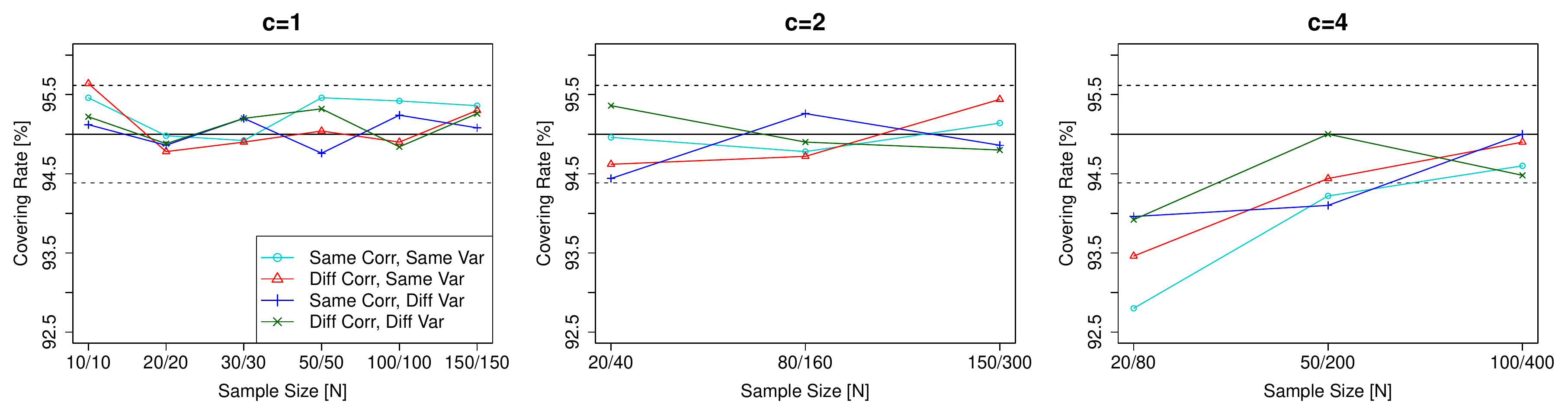}
	\caption{Simulation results for the SCBs of the difference between mean curves of two samples of smooth Gaussian processes. The solid black line is the targeted level of the SCBs and the dashed black line are twice the standard error for a Bernoulli random variable with $p=0.95$.}
	  \label{fig:2sampleSCBresults}
  \end{figure} 
  
  \FloatBarrier  
  \section{Application to DTI fibers}\label{sec:DTI}
   \FloatBarrier
 We apply our method to a data set studying the impact of the eating disorder anorexia nervosa on white-matter tissue
 properties during adolescence in young females. The experimental setup and a different methodology to statistically analyze the data using pointwise testing and permutation tests can be found in the original article \cite{Travis2015}. The data set consists of $15$ healthy volunteers -- a control group -- and $15$ patients. For each volunteer $27$ different neural fibers were extracted.
 
In order to locate differences in the DTI fibers we use two-sample $95\%$-SCBs for the difference in the population mean between the control group and the patients for each fiber over the domain $\cS=[0,100]$. Robustness of detected differences across scales is tested by computing also the SCBs for a difference between the population mean of the scale processes as proposed in Section \ref{sec:ObsNoise}. For the latter we used a Gaussian smoothing kernel  and the considered bandwidth range $\cH=[1.5,10]$ is sampled at $200$ equidistant bandwidths.

The results for the three fibers, where we detect significant differences, are shown in Figure \ref{fig:ResultsDTI}. Our results are mostly consistent with the results from \citep{Travis2015} in the sense that we also detect significant differences in the right thalamic radiation and the left SLF. We also detect differences in the  right cingulum hippocampus, which Travis et al did not detect. However, they also found differences for other fibers. Here we have to note that they used several criteria for claiming a significant difference, which did not take simultaneous testing and the correlation structure of this data into account and therefore might be false positives.  
    \begin{figure}\centering
		\includegraphics[width=1\textwidth]{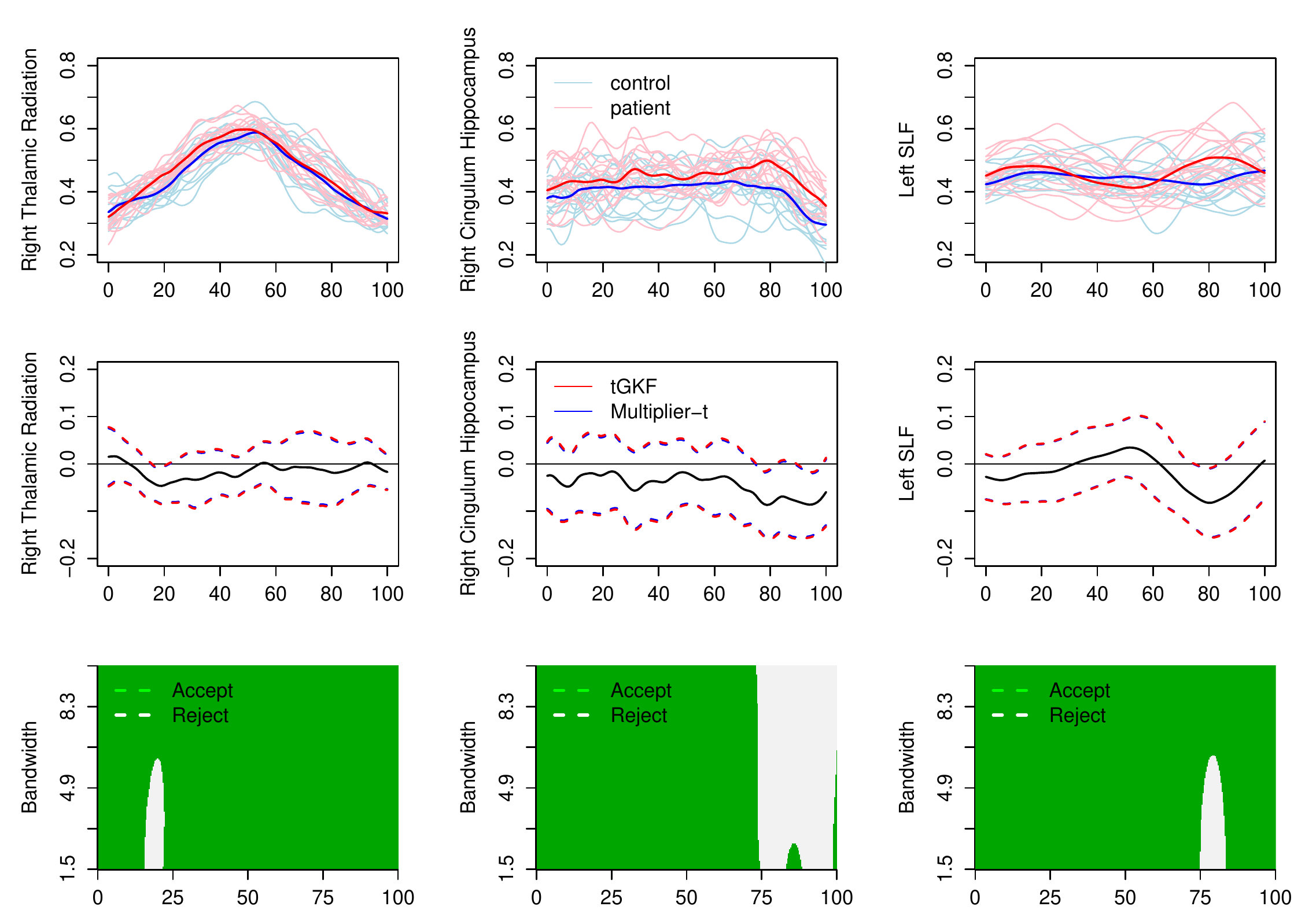}
	\caption{ DTI fibers with significant differences in the mean function. \textbf{Top row:} data and sample means. \textbf{Middle row:} estimated difference of the mean function and $95\%$-SCBs constructed using the tGKF and the Multiplier-$t$. \textbf{Bottom row:} areas where the confidence bands for the difference in the mean function for the scale process do not include 0.}
	  \label{fig:ResultsDTI}
  \end{figure}
 
 \paragraph{Acknowledgments}
  The authors are very thankful to Alex Bowring and Thomas Nichols for proposing to try Rademacher multipliers instead of Gaussian multipliers, which stimulated some further exploration on the multipliers and improved the performance of the bootstrap method considerably. We are also grateful for the in detail proof reading of Sam Davenport and especially for detection of an error in Theorem $2$. 
  
  \FloatBarrier
  \appendix
  \section{Proofs}
      \subsection{Proof of Proposition \ref{prop:redundancy}}
      Assume that $Z$ is $(\cL^p, \delta)$-Lipschitz, then using convexity of $\vert \cdot \vert^p$ we compute
      \begin{align*}
       \EE\left[ \max_{s\in\cS} \vert Z(s) \vert^p \right] &\leq 2^{p-1}\EE\left[ \max_{s\in\cS} \vert Z(s)-Z(s') \vert^p \right] + 2^{p-1} \EE\left[ \max_{s\in\cS} \vert Z(s') \vert^p \right] \\
       &\leq
2^{p-1}\EE\left[ \vert A\vert ^p \right]\max_{s\in\cS} \delta(s,s')^p + 2^{p-1} \EE\left[ \vert Z(s') \vert^p \right] <\infty\,.
      \end{align*}
      Hence $Z$ has also a finite $p$th $\cC(\cS)$-moment.
      
	  \subsection{Proof of Claim in Remark \ref{rem:GKFassumptions}}
	  Using the multivariate mean value theorem and the Cauchy-Schwarz inequality yields
	  \begin{equation*}
		\Big\vert D^{(d,l)}\cG(s) - D^{(d,l)}\cG(s') \Big\vert^2 \leq \max_{t\in \cS} \big\Vert \nabla D^{(d, l)}\cG(t) \big\Vert^2   \Vert s-s'\Vert^2\,.
	  \end{equation*}
	  Applying the expectation to both sides and then taking the maximum of the resulting sums we obtain
	  \begin{equation*}
		\EE\left[ \big\vert  D^{(d,l)}\cG(s) - D^{(d,l)}\cG(s') \big \vert^2\right] \leq D \max_{k=1,...,D} \EE\left[ \max_{t\in \cS} \big\vert D^{(d, l,k)}\cG(t) \big\vert^2  \right] \Vert s-s'\Vert^2\,.
	  \end{equation*}
	  The proof follows now from the following two observations. Firstly, by Remark \ref{rem:boundedmoment} each of the expectations we take the maximum of is finite, since all components of the gradient of $\nabla D^{(d,l)}\cG$ are Gaussian processes with almost surely continuous sample paths. Secondly, $\vert \log\Vert x \Vert \vert^{-2} \geq x^2$ for all $0 < x<1$.

	  \subsection{Proof of Theorem \ref{thm:consistencyquantile}}
	  \begin{lemma}
	\label{lemma:inverse-delta-method}
	Let the function $\hat{f}_{N}(u)$ and its first derivative $\hat{f}'_{N}(u)$ be uniformly consistent estimators of the function $f(u)$ and its first derivative $f'(u)$, respectively, where both are uniformly continuous over $u\in\R$. Assume there exists an open interval $I=(a, b)$ such that $f$ is strictly monotone on $I$ and there exists a unique solution $u_0\in I$ to the equation $f(u)=0$. Define $\hat{u}_{N} = \sup \{u\in I: \hat{f}_{N}(u) = 0\}$. Then $\hat{u}$ is a consistent estimator of $u_0$.
\end{lemma}

\begin{proof}
        Assume w.l.o.g. that $f$ is strictly decreasing on $I$. Thus, for any $\varepsilon>0$ we have $f(u_0-\varepsilon)>0>f(u_0+\varepsilon)$ by $f(u_0)=0$. The assumption that $\hat{f}_{N}(u)$ is a consistent estimator of $f(u)$ yields
		\begin{equation*}\label{g-hat-root}
            \Prb\left( \hat{f}_{N}(u_0-\varepsilon)>0>\hat{f}_{N}(u_0+\varepsilon) \right) \to 1,
		\end{equation*}
		which implies that with probability tending to 1, there is a root of $\hat{f}$ in $I_{0,\varepsilon} = (u_0-\varepsilon, u_0+\varepsilon)$. On the other hand the monotonicity of $f$ guarantees the existence of an $\delta>0$ such that $\inf\{ |f(u)|:~u\in I\setminus I_{0,\varepsilon}\} > \delta$.
		Moreover, by the uniform consistency of $\hat{f}$, we have that
		\[
            \Prb\left( \sup_{u\in I} |\hat{f}_{N}(u) - f(u)|< \delta/2 \right) \to 1.
		\]
		Therefore, using the inequality
		\[
            \inf_{u\in I\setminus I_{0,\varepsilon}} |\hat{f}_{N}(u)|
            \ge \inf_{u\in I\setminus I_{0,\varepsilon}} |f(u)| -
            \sup_{u\in I\setminus I_{0,\varepsilon}} |\hat{f}_{N}(u) - f(u)|,
		\]
		we can conclude that
		\begin{equation*}\label{g-hat-no-root}
		\begin{aligned}
            \Prb\left( \inf_{u\in I\setminus I_{0,\varepsilon}} |\hat{f}_{N}(u)| > \delta/2 \right) 
            &\ge \Prb\left( \inf_{u\in I\setminus I_{0,\varepsilon}} |f(u)| -
            \sup_{u\in I\setminus I_{0,\varepsilon}} |\hat{f}_{N}(u) - f(u)| > \delta/2 \right) \\
            &= \Prb\left( \sup_{u\in I\setminus I_{0,\varepsilon}} |\hat{f}_{N}(u) - f(u)| < \inf_{u\in I\setminus I_{0,\varepsilon}} |f(u)| - \delta/2 \right)
            \to 1\,,
		\end{aligned}
		\end{equation*}
		which implies that with probability tending to 1, there is no root of $\hat{f}_{N}$ outside $I_{0,\varepsilon}$. Hence from the definition of $\hat{u}_{N}$, it is clear that $\hat{u}_{N}$ is the only root of $\hat{f}_{N}$ in $I$ with probability tending to 1. As an immediate consequence we have that
		\[
		\Prb[|\hat{u}_{N}-u_0|<\varepsilon] = \Prb[\hat{u}_{N} \in I_{0,\varepsilon}] \to 1,
		\]
		which finishes the proof that $\hat{u}$ is a consistent estimator of $u_0$.
\end{proof}

	  \begin{lemma}\label{lemma:ConvergenceQuantile}
		    Let $\hat  L_d^N$ be a consistent estimator of $L_d$ and $\mathrm{EEC}_{\cG}(u)$ given in equation \eqref{eq:ualphaGauss}.
		    \begin{enumerate}
		      \item $\left\Vert \mathrm{EEC}_{\cG}(u) -  \widehat{\mathrm{EEC}}_{t_{N-1}}(u) \right\Vert_\infty\xrightarrow{N\rightarrow\infty} 0$ almost surely.
		      \item $\left\Vert \mathrm{EEC'}_{\cG}(u) -  \widehat{\mathrm{EEC'}}_{t_{N-1}}(u) \right\Vert_\infty\xrightarrow{N\rightarrow\infty} 0$ almost surely.
	    \end{enumerate}
	  \end{lemma}
	  \begin{proof}
			Part 1. is a direct consequence of the consistency of the LKC estimates and the observation that the EC densities $\rho^{t_\nu}$ of a ${t_\nu}$-process with $\nu=N-1$ degrees of freedom converges uniformly to the EC densities of a Gaussian process $\rho^\cG$ as $N$ tends to infinity, i.e. $$\lim_{\nu\rightarrow\infty}\max_{u\in\R} \vert \rho^{t_\nu}(u)-\rho^\cG(u) \vert=0.$$
			The latter follows from \citet[Theorem 5.4]{Worsley1994}, which implies that the uniform convergence of EC densities is implied by the uniform convergence of
			\begin{equation*}
				    \lim_{\nu\rightarrow\infty} \max_{u\in\R} \left\vert \left(1+\tfrac{u^2}{\nu}\right)^{-\tfrac{\nu-1}{2}} - e^{-\tfrac{u^2}{2}} \right\vert = 0\,.
			\end{equation*}
			To see this, note that the distance
			\begin{equation*}
				      h_\nu(u) = \left(1+\tfrac{u^2}{\nu}\right)^{-\tfrac{\nu-1}{2}} - e^{-\tfrac{u^2}{2}} \geq 0\,,~~\text{ for }u\in\R
			\end{equation*}
			fulfills $\lim_{u\rightarrow\pm\infty} h_{\nu}(u)=0$. Thus, there is a $C_\nu = \max_{u\in\R} \vert h_\nu(u)\vert $ by continuity of $h_\nu$. Moreover, note that $h_\nu(u)\geq h_{\nu+1}(u)$ for $\nu\geq1$, all $u\in\R$ and $\lim_{\nu\rightarrow\infty}h_\nu(u)=0$. Hence, $C_\nu$ converges to zero for $\nu\rightarrow\infty$.
			
			 Part 2. uses the same ideas. The only difference is that we have to show that the function
			\begin{equation*}
				      h_\nu(u) = \tfrac{\nu-1}{\nu}\left(1+\tfrac{u^2}{\nu}\right)^{-\tfrac{\nu+1}{2}} - e^{-\tfrac{u^2}{2}}\,,~~\text{ for }u\in\R
			\end{equation*}			
			converges to zero uniformly in $u$.
	    \end{proof}
	    
	    \begin{proof}[Proof of Theorem \ref{thm:consistencyquantile}]
 				In order to prove the almost sure convergence $\hat q_{\alpha,N}\xrightarrow{N\rightarrow\infty}\tilde q_\alpha$, we only have to note that for $u$ large enough $u\mapsto\mathrm{EEC}_{\cG}(u)$ is strictly monotonically decreasing and therefore the combination of Lemma \ref{lemma:inverse-delta-method} and \ref{lemma:ConvergenceQuantile} yields the claim.  
	    \end{proof}

	  \subsection{Proof of Theorem \ref{thm:AsymptoticValidity}}
        Note that by \cite[Theorem 4.3]{Taylor2005} we have for $\mathcal{Z}$ a mean zero Gaussian process over a parameter set $\mathcal{T}$ that
        \begin{equation*}
            \liminf_{u\rightarrow\infty} -u^2\log\left\vert \Prb\left( \max_{t\in \mathcal{T}} \mathcal{Z}(t)\geq u \right) - \mathrm{EEC}_{\mathcal{Z}}(u) \right\vert \geq \frac{1}{2}\left( 1+ \frac{1}{2\sigma_c^2} \right)\,,
        \end{equation*}
        where $\sigma_c^2$ is a variance depending on an associated process to $\mathcal{Z}$.  This implies that there is a $\tilde u$ such that for all $u\geq \tilde u$ we have that
        \begin{equation}\label{eq:ECHequation}
                \left\vert \Prb\left( \max_{t\in \mathcal{T}} \mathcal{Z}(t) \geq u \right) -  \mathrm{EEC}_{\mathcal{Z}}(u)  \right\vert \leq e^{-\big(\tfrac{1}{2}+\tfrac{1}{2\sigma_c^2}\big) u^2 }\,.
        \end{equation}		  
		Equipped with this result using the definition $M=\max_{s\in \cS}\cG(s)$ and $\vert M \vert=\max_{s\in \cS}\vert \cG(s)\vert$ we compute
        \begin{align*}
			      \Big\vert 1-&\alpha-\Prb\big( ~\forall s\in \cS:~\eta(s)\in SCB(s,\hat q_{\alpha,N} ) \big) \Big\vert \leq \left\vert -\alpha+ \Prb\Bigg( \max_{s\in \cS} \left\vert \tau_N \frac{\hat\eta_N(s) - \eta(s)}{\hat\varsigma_N(s)} \right\vert > \hat q_{\alpha,N} \Bigg) \right\vert\\
			      &\leq \left\vert \Prb\Bigg( \vert M \vert> \hat q_{\alpha,N} \Bigg) -\alpha \right\vert + \left\vert \Prb\Bigg( \max_{s\in \cS} \left\vert \tau_N \frac{\hat\eta_N(s) - \eta(s)}{\hat\varsigma_N(s)} \right\vert > \hat q_{\alpha,N} \Bigg) - \Prb\Bigg( \vert M \vert> \hat q_{\alpha,N} \Bigg) \right\vert = I + I\!I\,.
        \end{align*}
        Here $I\!I$ converges to zero for $N$ tending to infinity by the fCLT for $\hat\eta_N$ and the consistent estimation of $\hat\varsigma_N$ from \textbf{(E1-2)}. Therefore it remains to treat $I$.

		To deal with this summand, note that
		  \begin{equation*}
		    \Prb\Bigg( \vert M \vert > \hat q_{\alpha,N} \Bigg) = \Prb\Bigg( \max_{s\in \cS}\vert \cG(s)\vert^2> \hat q_{\alpha,N}^2 \Bigg) = \Prb\Bigg( \max_{(s,v)\in \cS\times S^0} \mathcal{Z}(s,v) > \hat q_{\alpha,N} \Bigg)\,,
		  \end{equation*}
            where the Gaussian random process $\mathcal{Z}$ over $\mathcal{T}=\cS\times S^0$, where $S^0=\{1,-1\}$, is defined by $\mathcal{Z}(s,v) = v\cdot\mathcal{G}(s)$.
            
            Using the above equality, $\widehat{\mathrm{EEC}}_{t_{N-1}}(\hat q_{\alpha,N}) =\alpha/2$, i.e. the definiton of our estimator $\hat q_{\alpha,N}$ from Equation \eqref{eq:ualpha} and Lemma \ref{lemma:ConvergenceQuantile} we have that
			\begin{align*}
                    I = &\left\vert \Prb\Bigg( \max_{s\in \cS\times S^0} \mathcal{Z}(s,v) > \hat q_{\alpha,N} \Bigg) - 2\widehat{\mathrm{EEC}}_{t_{N-1}}(\hat q_{\alpha,N}) \right\vert \\
                    &~~ ~ ~ ~ ~ ~ ~~ ~ ~ ~ ~ ~ ~ ~ ~ ~ ~ ~ ~\xrightarrow{N\rightarrow\infty} \left\vert \Prb\Bigg( \max_{(s,v)\in \cS\times S^0} \mathcal{Z}(s,v)> \tilde q_\alpha \Bigg) - 2\mathrm{EEC}_{\cG}(\tilde q_\alpha) \right\vert. 
			\end{align*}
            Thus, using the fact that $L_d(\mathcal{T},\mathcal{Z}) = L_0(S^0)L_d(\cS,\cG) =2L_d(\cS,\cG)$ and \eqref{eq:ECHequation} and the observation that $\tilde q_\alpha$ is monotonically increasing in $\alpha$ for $\alpha$ small enough, we can bound $I$ by
            \begin{equation*}
                I = \left\vert \Prb\Bigg( \max_{(s,v)\in \cS\times S^0} \mathcal{Z}(s,v)> \tilde q_\alpha \Bigg) - \mathrm{EEC}_{\mathcal{Z}}(\tilde q_\alpha) \right\vert \leq e^{-\tfrac{1}{2+2\sigma_c^2}  \tilde q_\alpha^2 }
            \end{equation*}
            for all $\alpha$ smaller than some $\alpha'$, which finishes the proof.

	  \qed
	  \begin{remark}
	   The specific definition of $\sigma_c$ associated with the Gaussian process $\mathcal{Z}$ can be found in \cite{Taylor2005}.
	  \end{remark}

    \subsection{Proofs of Theorem \ref{thm:GenericAsymptoticCBs} and \ref{thm:GenericAsymptoticCBdiffs}}
    The following Lemma provides almost sure uniform convergence results and will be used often in the following proofs.
          \begin{lemma}\label{Lem:UniformConvergence}
		Assume that $X$ and $Y$ are $(\cL^1, \delta)$-Lipschitz processes. Let $X_1,...,X_N\iid X$ and $Y_1,...,Y_N\iid Y$ be two samples. Then $\overline{\textbf{X}} \xrightarrow{N\rightarrow\infty} \EE[X]$ uniformly almost surely. If $X$ and $Y$ are $(\cL^2, \delta)$-Lipschitz processes with finite second $\cC(\cS)$-moments. Then $\widehat{\rm cov}_N \!\left[ \textbf{X},  \textbf{Y}\, \right]  \xrightarrow{N\rightarrow\infty} {\rm cov}\!\left[ X,  Y\, \right]$ uniformly almost surely.
	\end{lemma}
      \begin{proof}
		  \emph{First claim:} Using the generic uniform convergence result \citet[Theorem 21.8]{Davidson1994}, we only need to estalish strong stochastical equicontinuity (SSE) of the random function $\overline{\textbf{X}} - \EE[X]$, since pointwise convergence is obvious by the SLLNs. SSE, however, can be easily established using \cite[Theorem 21.10 (ii)]{Davidson1994}, since
		  \begin{equation*}
			\left\vert N^{-1} \sum_{n=1}^N \big( X_n(s)-X_n(s') \big) - \EE[X(s)-X(s')] \right\vert \leq \left( N^{-1} \sum_{n=1}^N A_n + \EE[A]\right)\delta(s,s') =C_N\delta(s,s')
		  \end{equation*}
		  for all $s,s'\in\cS$. Here $A_1,...,A_N\iid A$ denote the random variables from the $(\cL^1, \delta)$-Lipschitz property of the $X_n$'s and $X$ and hence the random variable $C_N$ converges almost surely to the constant $2 \EE[A]$ by the SLLNs.
		  
		  \emph{Second claim:} Adapting the same strategy as above and assuming w.l.o.g. $\EE[X]=\EE[Y]=0$, we compute
		  \begin{align*}
		   \left\vert \frac{1}{N} \sum_{n=1}^N X_n(s)Y_n(s) - X_n(s')Y_n(s') \right\vert
		  				&\leq \left(\frac{1}{N} \sum_{n=1}^N M(X_n)B_n+M(Y_n) A_n \right)\delta(s,s')\\
		  				&\leq \left( \sqrt{ \sum_{n=1}^N \tfrac{M(X_n)^2}{N} }\sqrt{ \sum_{n=1}^N \tfrac{B_n^2}{N} }+ \sqrt{ \sum_{n=1}^N \tfrac{M(Y_n)^2}{N} }\sqrt{ \sum_{n=1}^N \tfrac{A_n^2}{N} } \right)\delta(s,s')\,,
		  \end{align*}
		  where $M(X) = \max_{s\in\cS} \vert X(s)\vert$ and $B_1,...,B_N\iid B$ denote the random variables from the $(\cL^2, \delta)$-Lipschitz property of the $Y_n$'s and $Y$. Again by the SLLNs the random Lipschitz constant converges almost surely and is finite, since $X$ and $Y$ have finite second $\cC(\cS)$-moments and are $(\cL^2,\delta)-$Lipschitz.
	\end{proof}
	\begin{lemma}\label{lemma:KLexpansion}
		  Let $\fc$ be a covariance function. Then
		  \begin{enumerate}
		   \item[(i)] If $\fc$ is continuous and has continuous partial derivatives up to order $K$, then the Gaussian process $\cG(0,\fc)$ has $\cC^K$-sample paths with almost surely uniform and absolutely convergent expansions
		   		  \begin{equation}
					   D^I\!\cG(s) =  \sum_{i=1}^\infty \sqrt{\lambda_i} A_i D^I\!\varphi_i(s)\,,
				  \end{equation}
				  where $\lambda_i, \varphi_i$ are the eigenvalues and eigenfunctions of the covariance operator of $Z$ and $\{A_i\}_{i\in\N}$ are i.i.d. $\cN(0,1)$.
		    \item[(ii)] If $Z$ and  all its partial derivatives $D^I\!Z$ with multi-indices satisfying $\vert I \vert \leq K$, $K\in \N$, are $(\cL^2,\delta)-$Lipschitz processes with finite $\cC(\cS)$-variances, then $\fc$ is continuous and all partial derivatives $D^I_s\!D^{I'}_{s'}\!\fc(s,s')$ for $I,I'\leq K$ exist and are continuous for all $s,s'\in\cS$.
		  \end{enumerate}
	\end{lemma}
	\begin{proof}

	    (i) Since $\fc$ is continuous the process $\cG(0,\fc)$ is mean-square continuous. Hence there is a Karhunen-Lo\'eve expansion of the form
	    \begin{equation}
			\cG(s) =  \sum_{i=1}^\infty \sqrt{\lambda_i} A_i \varphi_i(s)\,,
	    \end{equation}
	    with $\lambda_i, \varphi_i$ are the eigenvalues and eigenfunctions of the covariance operator associated with $\fc$ and $\{A_i\}_{i\in\N}$ are i.i.d. $\cN(0,1)$. From \citet[Theorem 5.1]{Ferreira2012} we have that $D^{I}\varphi\in \cC^K(\cS)$. Moreover, it is easy to deduce from their equation (4.3) that
	    \begin{equation}
			\cG(s) =  \sum_{i=1}^\infty \sqrt{\lambda_i} A_i D^{I}\varphi(s)\,,
	    \end{equation}	    
	    is almost surely absolutely and uniformly convergent. Note that the assumption that $\fc\in\cC^{2s}(\cS\times\cS)$ is too strong in their article. They in fact only require in their proofs that all partial derivatives $D^{I}_sD^{I'}_{s'}\fc(s,s')$ for $I,I'<s$ exist and are continuous. 
	    
	    	    (ii): The continuity is a simple consequence of the $(\cL^2,\delta)-$Lipschitz property and the finite $\cC(\cS)$-variances. Let $X$ be a process with these properties, then using the Cauchy-Schwarz inequality
	    \begin{align*}
		  \vert\fc(s,t)-\fc(s',t')\vert &\leq \left\vert \EE\left[ (X_s-X_{s'})X_t + (X_t-X_{t'})X_{s'} \right] \right\vert \\
								    &\leq \EE\left[ \vert X_s-X_{s'}\vert \vert X_t\vert\right] + \EE\left[ \vert X_t-X_{t'}\vert \vert X_{s'}\vert \right]\\
								    &\leq  \sqrt{\EE\left[ (X_s-X_{s'})^2\right]} \sqrt{\EE\left[  \max_{t\in\cS}  X_t^2\right]} + \sqrt{\EE\left[ ( X_t-X_{t'} )^2\right]} \sqrt{\EE\left[  \max_{s'\in\cS} X_{s'}^2 \right]}\\
								    &\leq C \left( \delta(s,s') + \delta(t,t') \right)
	    \end{align*}
	    for some $C<\infty$ and therefore $\fc$ and the covariances of $D^I\!Z$ are continuous. We only show that  $D^d_s\fc(s,s')$ exists and is continuous. The argument is similar for the higher partial derivatives. From the definition we obtain for all $s,s'$
	    \begin{align*}
		  \lim_{h\rightarrow0} h^{-1}\big( \fc(s,s')-\fc(s+he_d,s') \big) &=  \lim_{h\rightarrow0} \EE\left[ h^{-1}\big( Z(s)-Z(s+he_d) \big) Z(s')\right]\,,
	    \end{align*}
	    where $e_d$ denotes the $d$-th element of the standard basis of $\R^D$. Thus, we only have to prove that we can interchange limits and integration. The latter is an immediate consequence of Lebesgue's dominated convergence theorem, where we obtain the $\cL^1$ majorant from the $(\cL^2,\delta)-$Lipschitz property as $AZ(s')$, where $A\in \cL^2$.
	\end{proof}

	\begin{proof}[Proof of Theorem \ref{thm:GenericAsymptoticCBs}]
		  (i) Since $Z$ is $(\cL^2, \delta)$-Lipschitz the main result in \citet{Jain1975} immediatly implies $\textbf{(E1)}$ with $\tau_N=\sqrt{N}$ and $\fr=\fc$. Condition \textbf{(E2)} is obtained from the second part of Lemma \ref{Lem:UniformConvergence}, since $\sigma(s) Z(s)$ is $(\cL^2, \delta)$-Lipschitz and has finite second $\cC(\cS)$-moment.
		  
		  (ii) We only need to show that the Gaussian limit process with the covariance $\fc$ fulfills \textbf{(G1)} and \textbf{(G3)}. Note that condition \textbf{(G3)} is a consequence of \textbf{(G1)} and the $\cC^3$-sample paths by Remark \ref{rem:GKFassumptions}. But \textbf{(G1)} is already a consequence of Lemma \ref{lemma:KLexpansion}.
      \end{proof}
      
      \begin{lemma}\label{lemma:G2assumption}
	    Let $Z$ fulfill the assumptions of Theorem \ref{thm:GenericAsymptoticCBs} (ii) except for \textbf{(G2)} and for all $d,l\in\{1,...,D\}$ suppose that $\cov\!\left[ \left( D^d\! Z(s) , D^{(d,l)}\!Z(s)\right) \right]$ has full rank for all $s$. Then $\cG=\cG(0,\fc)$ fullfills \textbf{(G2)}.
      \end{lemma}
      \begin{proof}
		  Using the series expansions from Lemma \ref{lemma:KLexpansion} we have that for multi-indices $I_1,..., I_P$, $K\in\N$ and all $v\in \R^P$ it follows that
		   \begin{equation}
			  \left( D^{I_1}\cG(s),..., D^{I_K}\cG(s) \right) v^T
				    =\sum_{i=1}^\infty \sqrt{\lambda_i} A_i \sum_{p=1}^P v_p D^{I_p}\!\varphi_i(s)
		   \end{equation}
		   is convergent for all $s$ (even uniformly). Note that we used here that the expansions are absolutely convergent such that we can change orders in the infinite sums. Thus, it is easy to deduce that $ \left( D^{I_1}\cG,..., D^{I_P}\cG \right)$ is a Gaussian process.
	      
		  Therefore $\big( D^d\cG(s) , D^{(d,l)}\cG(s) \big)$ is a multivariate Gaussian random variable for all $s\in\cS$, which is nondegenerate if and only if its covariance matrix is non-singular. But this is the case by the assumption, since it is identical to the covariance matrix $\cov\!\left[ \left( D^d\! Z(s) , D^{(d,l)}\!Z(s)\right) \right]$.  
      \end{proof}
	\begin{proof}[Proof of Theorem \ref{thm:GenericAsymptoticCBdiffs}]      
	 The proof is almost identical to the proof of Theorem \ref{thm:GenericAsymptoticCBs} and therefore omitted.
	\end{proof}

      \subsection{Proof of Theorem \ref{thm:CovarianceConsistency}}
       First note that using the definition of $R_n$ from Equation \ref{def:normedresiduals} we obtain
	\begin{equation*}
		     D^d\! \big( R_n \big) =  D^d\! \big(\tfrac{\mu - \hat\mu}{\hat\sigma_N} \big) +  D^d\!\big(\tfrac{\sigma}{\hat\sigma_N}\big)Z_n + \tfrac{\sigma}{\hat\sigma_N}D^d\! Z_n\,.
	\end{equation*}
	Thus, the entries of the sample covariance matrix $\hat\Lambda_{N}$ are given by
	\begin{align}
	  \begin{split}\label{eq:CovDecompostion}
		  \widehat{\cov}\!\left[ D^d \textbf{R}, D^l \textbf{R} \right] =
			  &~~~ \widehat{\var} \big[ \textbf{Z} \big]   D^d\!\big(\tfrac{\sigma}{\hat\sigma_N}\big)D^l\!\big(\tfrac{\sigma}{\hat\sigma_N}\big)
			  + \widehat{\cov} \left[ \textbf{Z}, D^d\textbf{Z} \right] \tfrac{\sigma}{\hat\sigma_N}   D^l\!\left(\tfrac{\sigma}{\hat\sigma_N}\right)\\
			  &+ \widehat{\cov} \left[ \textbf{Z}, D^l\textbf{Z} \right] \tfrac{\sigma}{\hat\sigma_N}   D^d\!\left(\tfrac{\sigma}{\hat\sigma_N}\right)
			  +\widehat{\cov} \left[ D^d\textbf{Z}, D^l\textbf{Z} \right] \tfrac{\sigma^2}{\hat\sigma_N^2}
	  \end{split}
	\end{align}
	Now, the second part of Lemma \ref{Lem:UniformConvergence} applied to $Z$ and $D^l\! Z$ together with the Assumptions \textbf{(L)} and \textbf{(E2)} imply that the first three summands on the r.h.s. converge to zero almost surely in $\cC(\cS)$. Therefore, applying again Lemma \ref{Lem:UniformConvergence} to the remaining summand, we obtain
	\begin{equation*}
		    \widehat{\cov}_N\!\left[ D^d\textbf{R}, D^l \textbf{R} \right]\xrightarrow{N\rightarrow\infty}  \cov\left[  D^d\!Z, D^l\!Z \right]=\Lambda_{dl}
	\end{equation*}
	uniformly almost surely. Thus, $\hat\Lambda_N\rightarrow\Lambda$ uniformly almost surely.
	
	Assume that $X,Y$ are $(\cL^4, \delta)$-Lipschitz with finite forth $\cC(\cS)$-moments, then
	\begin{align*}
	 \vert X(s)Y(s) - X(s')Y(s') \vert \leq \big( M(X)A+ M(Y)B \big)\delta(s,s')
	\end{align*}
	with $A,B$ the random variables in the $(\cL^4, \delta)$-Lipschitz property of $Y,X$. Note that
	\begin{align*}
	 \EE\left[ \big( M(X)A+ M(Y)B \big)^2\right] &\leq 2\EE\left[ (M(X)A)^2 + (M(Y)B)^2\right]\\
	 											 &\leq 2 \sqrt{\EE[M(X)^4]\EE[A^4]} + 2\sqrt{\EE[M(Y)^4]\EE[B^4]} <\infty\,.
	\end{align*}
	by $(a+b)^2\leq 2a^2+2b^2$ for all $a,b\in\R$ and the Cauchy-Schwarz inequality.
	Thus, a sample $X_1Y_1,..., X_NY_N\iid XY$ fulfills the assumptions for the CLT in $\cC(\cS)$ given in \cite{Jain1975}. Therefore, the following sums converge to a Gaussian process in $\cC(\cS)$:
	\begin{equation*}
		    \sqrt{N} \Big( \widehat{\var} \big[ \textbf{Z} \big] \Big)\,,~ ~ ~\sqrt{N}\widehat{\cov} \big[ \textbf{Z}, D^l\textbf{Z} \big]\,,~ ~ ~ \sqrt{N}\widehat{\cov} \left[  D^d\textbf{Z},  D^l\textbf{Z} \right] \text{ for }d,l=1,...,D\,.
	\end{equation*}
	Thus, using the latter together with equation \eqref{eq:CovDecompostion} and the Assumptions \textbf{(L)} and  \textbf{(E2)} we obtain
	\begin{equation*}
		    \sqrt{N} \Bigg(  \widehat{\cov}\!\left[ D^d\textbf{R}, D^l \textbf{R} \right] - \Lambda_{dl} \Bigg) \xRightarrow{N\rightarrow\infty}\cG\big( 0, \mathfrak{t}_{dl}  \big)
	\end{equation*}
	This combined with the standard multivariate CLT yields the claim.
      \qed
         
      \subsection{Proof of Theorem \ref{thm:LKCConsistency} and Corollary \ref{cor:LKCConsistency}}
      \begin{proof}[Proof of Theorem \ref{thm:LKCConsistency}]
            The almost sure uniform convergence of $\hat\Lambda_N$ to $\Lambda(s)$ from Theorem \ref{thm:CovarianceConsistency} implies that the integrands of equations \eqref{eq:L1IntegralEstimator}, \eqref{eq:L1L2IntegralEstimator} are almost surely uniform convergent, since $\hat\Lambda_N$ is composed with a differentiable function. Thus, we can interchange the limit $N\rightarrow\infty$ and the integral everywhere except for a set with measure zero. This gives the claim.       
      \end{proof}
      \begin{proof}[Proof of Corollary \ref{cor:LKCConsistency}]
	    We only have to check \textbf{(L)} and \textbf{(E2)} hold true. Therefore note that
	    \begin{equation*}
		      \hat\sigma_N^2(s)/\sigma^2(s) = \widehat{\var} \left[ \sigma(s)\textbf{Z}(s) \right]/\sigma^2(s)= \widehat{\var} \left[ \textbf{Z}(s) \right] \text{ and }D^l\!\!\left(\hat\sigma_N^2(s)/\sigma^2(s)\right)=2\,\widehat{\rm cov} \left[ \textbf{Z}(s), D^l\textbf{Z}(s) \right]\,.
	    \end{equation*}
	      Thus, both uniform convergence results follow from Lemma \ref{Lem:UniformConvergence}.
      \end{proof}
      
      \subsection{ Proof of Theorem \ref{thm:CLTlkc} }
      We want to use the functional delta method, e.g., \citet[Theorem 2.8]{Kosorok2008}. By Theorem \ref{thm:CovarianceConsistency} the claim follows, if we prove that the corresponding functions are Hadamard differentiable and can compute this derivative.
      
      \emph{Case 1D:} We have to prove that the function
      \begin{equation*}
		H:~~~ \big( \cC(\cS), \Vert\cdot\Vert_\infty \big) \rightarrow \R\,,~~~ f \mapsto \int_\cS \sqrt{f(s)}\, ds
      \end{equation*}
      is Hadamard differentiable. Therefore, note that the integral is a bounded linear operator and hence it is Fr\'echet differentiable with derivative being the integral itself. Moreover, $f\mapsto\sqrt{f}$ is Hadamard differentiable by \citet[Lemma 12.2]{Kosorok2008} with Hadamard derivative $DH_f(\alpha) = 1/\sqrt{4f}\alpha$ tangential even to the Skorohod space $D(\cS)$. Combining this, we obtain the limit distribution $\sqrt{N}(\hat L_1^N-L_1)$ from the fCLT for $\hat\Lambda_N$ given in Theorem \ref{thm:CovarianceConsistency} to be distributed as
      \begin{equation}
	      D\!H_\Lambda(G) = \frac{1}{2}\int_\cS \frac{ G(s) }{\sqrt{\Lambda(s)} } \,ds\,,
      \end{equation}
      where $G(s)$ is the asymptotic Gaussian process given in Theorem \ref{thm:CovarianceConsistency}.
      
       \emph{Case 2D:} The strategy of the proof is the same as in 1D, i.e. we need to calculate the Hadamard (Fr\'echet) derivative of
      \begin{align*}
		H:~~~ &\big( \cC(\cS) \times \cC(\cS)\times \cC(\cS), \Vert\cdot\Vert_\infty \big) \rightarrow \R^2\,,\\~~~ (f_1,f_2,f_3) &\mapsto \left( \frac{1}{2}\int_0^1 \sqrt{ \tfrac{d\gamma}{dt}^T\!\!(t) \iota\Big(f_1\big(\gamma(t)\big),f_2\big (\gamma(t)\big),f_3\big(\gamma(t)\big)\Big) \tfrac{d\gamma}{dt}\!(t)}\,dt
		    , \int_\cS \sqrt{ \det\big( \Lambda(s)\big) } ds_1ds_2\, \right).
      \end{align*}
    The arguments are the same as before. Thus, using the chain rule and derivatives of matrices with respect to their components the Hadamard derivative evaluated at the process $\cG$ is given by
    \begin{align*}
          d H_\Lambda(G) = &\Bigg( \frac{1}{2}\int_0^1 \frac{1}{\sqrt{\tfrac{d\gamma}{dt}^T\!\!(t)\Lambda\big(\gamma(t)\big)\tfrac{d\gamma}{dt}\!(t)}} {\rm tr}\Bigg( \Lambda\big(\gamma(t)\big) \iota\Big( G\big(\gamma(t)\big) \Big) \Bigg) dt,\\ &~~~~~~~~~~~~~~~\int_{\cS} \frac{1}{\sqrt{\det\big(\Lambda(s)\big)}} {\rm tr}\Big( \Lambda(s) \iota\big( {\rm diag}(1,-1,1)G(s) \big) \Big) ds \Bigg)\,.
    \end{align*}
    
    \subsection{ Proof of Corollary \ref{cor:1DCLT} }
    Note that it is well-known that the covariance function of the derivative of a differentiable process with covariance function $\fc$ is given by $\dot{\fc}(s,s') = D^1_sD^1_{s'} \mathfrak{c}(s,s')$. Moreover, using the moment formula for multivariate Gaussian processes we have that
    \begin{align*}
	\cov\left[ \big(D^1_{s}Z(s)\big)^2, \big(D^1_{s'}Z(s')\big)^2 \right] &= \EE\left[ \Big(\big(D^1_{s}Z(s)\big)^2 - \dot{\fc}(s,s) \Big)\Big(\big(D^1_{s'}Z(s')\big)^2 - \dot{\fc}(s',s') \Big) \right] \\
	                   &= \EE\left[ \big(D^1_{s}Z(s)\big)^2 \big(D^1_{s'}Z(s')\big)^2\right] - \dot{\fc}(s,s)\dot{\fc}(s',s') \\
	                   &= \dot{\fc}(s,s)\dot{\fc}(s',s') + 2\dot{\fc}(s,s')- \dot{\fc}(s,s)\dot{\fc}(s',s') \\
	                   &= 2\dot{\fc}(s,s').
    \end{align*}
    Combining this with the observation that the variance of the zero mean Gaussian random variable $\frac{1}{2}\int_{\cS} \frac{G(s)}{\sqrt{\Lambda(s)}} ds$ is given by
    \begin{equation*}
	  \tau^2 = \frac{1}{4}\int_{\cS}\int_{\cS} \tfrac{\cov\big[(D^1_{s}Z(s))^2, (D^1_{s'}Z(s'))^2 \big]}{\sqrt{\Lambda(s) \Lambda(s')}} dsds'
    \end{equation*}
    yields the claim.
    
    \subsection{ Proof of Theorem \ref{thm:CLTscaleProcess} }
		   We want to apply \citet[Theorem 10.6]{Pollard1990}. Therefore, except for the indices we adapt the notations of that theorem and define the necessary variables. Recall that $\max_{s\in\cS}\sigma(s)\leq B<\infty$. We obtain
		   \begin{align*}
			  f_{Nn}(s,h) &= \frac{1}{\sqrt{N}P }\sum_{p=1}^{P} \left(\sigma(s_{p})Z_n(s_{p}) + \varepsilon_{np}\right)  K( s - s_p, h)\\
			  F_{Nn}     &= \sqrt{\frac{2}{N}} \left( B\max_{s\in \cS} \vert Z_n(s) \vert + \max_{s\in\cS}\vert\varepsilon_n(s)\vert \right)\\
			  X_N(s,h)      &= \sum_{n=1}^N  f_{Nn}(s,h)\,.
		   \end{align*}
		    We have to establish the assumptions (i), (iii) and (iv) as (v) is trivially satisfied in our case and (ii) is Assumption \eqref{assumption:covarianceExists}. As discussed in \citet[p.1759]{Degras2011} the manageability (i) follows from the inequality
		    \begin{align*}
			\vert  f_{Nn}(s,h) -  f_{Nn}(s',h') \vert &\leq \frac{1}{\sqrt{N}} \sqrt{ \frac{1}{P} \sum_{p=1}^P \big( \sigma(s_p)Z_n(s_p) +\varepsilon_n(s_p) \big)^2 } \sqrt{ \frac{1}{P} \sum_{p=1}^P \big(  K( s - s_p, h) -  K( s' - s_p, h') \big)^2 } \\
			  &\leq \sqrt{\frac{2}{N}} \left( B\max_{s\in\cS} \vert Z_n(s) \vert + \max_{s\in\cS}\vert \varepsilon_n(s)\vert \right)L\Vert (s,h) - (s',h') \Vert^\alpha\\
			  &= LF_{Nn}\epsilon\,,
		    \end{align*}
		    if $\Vert (s,h) - (s',h') \Vert<\epsilon^{1/\alpha}$. Assumption (iii) follows since we can compute
		    \begin{align*}
			    \sum_{n=1}^N \EE\left[ F_{Nn}^2 \right] = N\EE\left[ F_{N1}^2 \right]
			    \leq 4B^2\EE\left[ \max_{s\in\cS} \vert Z_1(s)\vert^2 \right] + 4\EE\left[ \max_{s\in\cS}\varepsilon_1^2(s) \right] <\infty
		    \end{align*}
		    and (iv) is due to
		    \begin{align*}
			    \sum_{n=1}^N \EE\left[ F_{Nn}^2 I(F_{Nn}>\epsilon) \right] = N\EE\left[ F_{N1}^2 I\big( \sqrt{N}F_{N1}>\sqrt{N}\epsilon\big) \right] \xrightarrow{N\rightarrow\infty}0
		    \end{align*}
		    for all $\epsilon>0$, which follows from the convergence theorem for integrals with monotonically increasing integrands and the fact that by Markov's inequality
		    \begin{align}
			    \EE\left[ I\big( \sqrt{N}F_{N1}>\sqrt{N}\epsilon\big) \right] = \Pr\left( \sqrt{N}F_{N1}>\sqrt{N}\epsilon \right) \leq \frac{\EE\left[\sqrt{N}F_{N1}\right]}{\sqrt{N}\epsilon} \xrightarrow{N\rightarrow\infty}0 \,,
		    \end{align}
		    for fixed $\epsilon>0$.
		    
		    The weak convergence to a Gaussian process now follows from \citet[Theorem 10.6]{Pollard1990}.
		    \qed
		    
		    \subsection{Proof of Proposition \ref{prop:ScaleSCBsufficientCond}}
		    The first step is to establish that for each $N$ the process
		    \begin{equation}
			\tilde Z(s,h) = \frac{1}{P}\sum_{p=1}^{P} \big(\sigma(s_p)Z(s_p)+\varepsilon(s_p)\big) K( s - s_p,h )\,,
		    \end{equation}
		   which has $\cC^3$-sample paths, has finite second $\cC(\cS)$-moment with a constant uniformly bounded over all $N$, and the process itself and its first derivatives are $(\cL^2, \delta)$-Lipschitz again uniformly over all $N$, since then the same arguments as in the proof of Corollary \ref{cor:LKCConsistency} and Lemma \ref{Lem:UniformConvergence} will yield \textbf{(L)} and \textbf{(E2)} and hence the consistency of the LKC estimation by Theorem \ref{thm:LKCConsistency}. Therefore, note that 
		    \begin{align*}
			\vert \tilde Z(s,h) \vert &\leq \frac{1}{P}\sum_{p=1}^{P} \vert\sigma(s_p)Z(s_p)+\varepsilon(s_p)\vert \cdot\vert K( s - s_p,h ) \vert \\
			&\leq \sqrt{ \frac{1}{P} \sum_{p=1}^P \big( \sigma(s_p)Z(s_p) +\varepsilon(s_p) \big)^2 } \sqrt{ \frac{1}{P} \sum_{p=1}^P \big(  K( s - s_p, h) \big)^2 }\\
			&\leq \left( B\max_{s\in\cS} \vert Z(s)\vert + \max_{s\in\cS} \vert \varepsilon(s)\vert \right)\vert K( s , h) \vert\,.
		    \end{align*}
		    This yields using $(a+b)^2 \leq 2(a^2+b^2)$ that
		    \begin{equation*}
		      \EE\left[ \max_{(s,h)\in\cS\times\cH} \vert \tilde Z(s,h) \vert^2 \right] \leq 2 \left( B^2\EE\left[ \max_{s\in\cS} \vert Z(s) \vert^2 \right] + C \right) \max_{(s,h)\in\cS\times\cH} K( s , h)^2 <\infty\,,
		    \end{equation*}
		    where the bound is independent of $N$. Basically, the same argument yields the $(\cL^2,\Vert \cdot\Vert)$-Lipschitz property for $\tilde Z(s,h)$ and all of its partial derivatives up to order $3$ with a bounding $\cL^2$ random variable independent of $N$.
		    
		    The differentiability of the sample paths of $\cG(0,\fr)$ follows from Lemma \ref{lemma:KLexpansion}(i).
  \FloatBarrier
  \bibliographystyle{apalike}
    \bibliography{references.tex}
\end{document}